\documentclass[a4paper, reqno, 11pt]{amsart}

\usepackage[english]{babel}
\usepackage{amsmath}
\usepackage{amssymb}
\usepackage{amsthm}
\usepackage{enumerate}
\usepackage{ifthen}
\usepackage{bbm}
\usepackage{color}
\provideboolean{shownotes} 
\setboolean{shownotes}{true}
\usepackage{hyperref} 
\usepackage{graphicx}
\usepackage{pgfplots}
\usepackage{tikz,tikz-3dplot} 
\usepackage{mathtools}
\usepackage{comment}
\usepackage{float}
\usepackage{geometry}
\usepackage[T1]{fontenc}
\usepackage{esint}
\newcommand{\margnote}[1]{
\ifthenelse{\boolean{shownotes}}%
{\marginpar{\raggedright\tiny\texttt{#1}}}%
{}%
}

\newcommand{\hole}[1]{
\ifthenelse{\boolean{shownotes}}%
{\begin{center} \fbox{ \rule {.25cm}{0cm}
\rule[-.1cm]{0cm}{.4cm} \parbox{.85\textwidth}{\begin{center}
\texttt{#1}\end{center}} \rule {.25cm}{0cm}}\end{center}}
{}
}
\usepackage{mathrsfs}
\newtheorem{thm}{Theorem}[section]

\newtheorem*{mainthm}{Theorem A}

\newtheorem*{mainthm2}{Theorem B}

\newtheorem*{mainthm3}{Theorem B'}

\newtheorem*{mainthm4}{Theorem C}

\newtheorem{prop}[thm]{Proposition}
\newtheorem{lem}[thm]{Lemma}
\newtheorem{cor}[thm]{Corollary}
\newtheorem{rem}[thm]{Remark}

\newtheorem{defn}[thm]{Definition}
\newtheorem*{assumption}{Assumption}

\newcommand{\e}{\varepsilon}		       
\newcommand{\R}{\mathbb{R}}
\newcommand{\T}{\mathbb{T}}
\newcommand{\N}{\mathbb{N}}
\newcommand{\Z}{\mathbb{Z}}
\newcommand{\dive}{\mathop{\mathrm {div}}}
\newcommand{\curl}{\mathop{\mathrm {curl}}}

\newcommand{\de}{\mathrm{d}}

\newcommand{\supp}{\mathrm{supp}\,}

\newcommand{\E}{\mathbb{E}}

\numberwithin{equation}{section}



\begin{document}

\title[Enhanced dissipation and magnetic reconnection]{Accelerated and fast magnetic reconnection through enhanced resistive dissipation for MHD equations}

\author[G. Ciampa]{Gennaro Ciampa}
\address[G.\ Ciampa]{DISIM - Dipartimento di Ingegneria e Scienze dell'Informazione e Matematica\\ Universit\`a  degli Studi dell'Aquila \\Via Vetoio \\ 67100 L'Aquila \\ Italy}
\email[]{\href{gciampa@}{gennaro.ciampa@univaq.it}}

\author[R. Luc\`a]{Renato Luc\`a}
\address[R. Luc\`a]{Institut Denis Poisson, UMR 7013, CNRS, 
Université d’Orléans, Bâtiment de Mathématiques, rue de
Chartres, F-45100 Orléans.}
\email[]{\href{rluca@}{renato.luca@univ-orleans.fr}}

\maketitle

\begin{abstract}
We consider the phenomenon of magnetic reconnection, namely a change in the topology of magnetic lines, for sufficiently regular solutions of the three-dimensional periodic magnetohydrodynamic (MHD) equations. We provide examples where magnetic reconnection occurs on time scales shorter than the resistive one, due to enhanced dissipation emerging from advective effects. This is the first analytical result where the advection term plays an active role in the reconnection process. A key aspect of our approach is a new estimate for enhanced diffusion of high Sobolev norms, which is of independent interest beyond its application to the MHD equations.
\end{abstract}

\section{Introduction}
In this work we analyze the Cauchy problem for the three-dimensional inviscid incompressible magnetohydrodynamic equations, which is the system of PDEs
\begin{equation}\label{eq:mhd}\tag{MHD}
\begin{cases}
\partial_t u_{\eta}+(u_{\eta}\cdot \nabla)u_{\eta}+\nabla p_{\eta}= (b_{\eta}\cdot \nabla)b_{\eta},\\
\partial_t b_{\eta}+(u_{\eta}\cdot \nabla)b_{\eta}=(b_{\eta}\cdot \nabla)u_{\eta}+\eta\Delta b_{\eta},\\
\dive u_{\eta}=\dive b_{\eta}=0,\\
u_{\eta}(0,\cdot)=u_{\eta}^\mathrm{in},\hspace{0.3cm} b_{\eta}(0,\cdot)=b_{\eta}^\mathrm{in},
\end{cases}
\end{equation}
where the parameter $\eta>0$ denotes the resistivity of the fluid, $u_{\eta}^\mathrm{in}, b_{\eta}^\mathrm{in}:\T^3\to\R^3$ are prescribed initial conditions, and the unknowns are the velocity field $u_{\eta}:[0,T]\times\T^3\to\R^3$, 
the magnetic field $b_{\eta}:[0,T]\times\T^3\to\R^3$, and the total pressure $p_{\eta}:[0,T]\times\T^3\to\R$. Here and in the sequel $\T^3:=\R^3/\Z^3$ denotes the three-dimensional torus. The equations are derived combining the Navier-Stokes equations of fluid dynamics with the Maxwell equations of electromagnetisms and model the motion of an electrically conducting incompressible inviscid fluid, see \cite{Landau}. The well-posedness of the system \eqref{eq:mhd} has been extensively investigated during the last decades, we refer to \cite{DL72, FMRR, Feff, Sch, Secchi, ST}, and references therein, for an overview of results ranging from the viscous and resistive to the ideal case.
In this paper we focus on a phenomenon of particular relevance to plasma physics: the {\em magnetic reconnection}. Magnetic reconnection refers to the process in which oppositely directed magnetic field lines within a plasma break and reattach, resulting in energy transfer and causing a change in the magnetic topology.  Mathematically, it can be defined as a change in the topological structure of the integral lines of the magnetic field $b_{\eta}$. Reconnection phenomena are observed  throughout the universe and play a key role in many dynamic processes in the solar corona, including solar flares, coronal mass ejections, the solar wind and coronal heating. For a physical introduction to the problem, we refer to \cite{MagnPhys, Priest, Schindler} and references therein. Although there is significant numerical and experimental evidence, the phenomenon remains not fully understood analytically.
In the non-resistive case ($\eta = 0$), the Alfven’s Theorem states that the magnetic lines are ``frozen'' into the fluid, meaning that the integral lines of a 
sufficiently regular magnetic field are transported by the fluid flow. More precisely, the flow induced by the velocity field is 
a diffeomorphism which maps integral lines of the magnetic field into integral lines of the magnetic field, at different times.  
The topological structure of the magnetic field is thus preserved under the evolution. The topological stability of the magnetic structure is related to the conservation of the magnetic helicity, which, in the non resistive case ($\eta = 0$), becomes a very subtle matter at low regularities, intimately related to anomalous dissipation phenomena.
We refer to \cite{BBV20,FL19,FLS20} for some very interesting (positive and negative) results in this direction. However, in the resistive case ($\eta>0$), the topology of the magnetic lines is expected to change, even for regular solutions.

Heuristically, one might expect a bound for the reconnection time like $T>1/\eta$, corresponding to a {\em resistive time scale}. However, it is both predicted and observed that the reconnection time is much shorter due to enhanced or anomalous diffusion effects induced by the advection term. For example, the Sweet-Parker model predicts a reconnection time of order $\eta^{-\frac12}$, while the model for resistive tearing instabilities predicts a reconnection time of order $\eta^{-\frac35}$, as detailed in \cite{Schindler}. 
{\color{red}}
The first rigorous mathematical construction of solutions exhibiting magnetic reconnection was provided in \cite{CCL} for the viscous model in both the two- and three-dimensional periodic setting. For a result on $\R^3$ we refer to \cite{LP23}. 
These examples were constructed by perturbing explicit solutions with topological properties that are {\em stable} under sufficiently regular perturbations. However, in the case for stance of \cite{CCL}, the reconnection time is proportional to $1/(\eta N^2)$, where $N$ is a characteristic length scale of the solution, see \cite[Remark 4.1]{CCL}. This leads to the following conclusions:
\begin{itemize}
\item[$i)$] the nonlinearity does not accelerate the process;
\item[$ii)$] even though the result holds for every $\eta>0$, the reconnection does not occur in a turbulent regime.
\end{itemize}

The purpose of this paper is to build upon the construction in \cite{CCL} by providing examples of {\em accelerated reconnection}. Specifically, we introduce the following definition.
\begin{defn}\label{def:fast rec}
A (sufficiently regular) family of solution $(u_{\eta},b_{\eta},p_{\eta})$ of \eqref{eq:mhd} 
shows accelerated magnetic reconnection if there exists $\eta_0\in(0,1)$ such that the following holds. For all $\eta \in (0, \eta_0)$ 
\begin{itemize}
\item  there exist $t_1 = t_1(\eta)$, $t_2 = t_2(\eta)$ such that there is no homeomorphism of $\T^3$ mapping the set of the integral lines of $b_{\eta}(t_1,\cdot)$ into integral lines of $b_{\eta}(t_2,\cdot)$;
\item there exists a continuous increasing function $\lambda_\mathrm{rec}: (0,\eta_0)\to(0,1)$ such that
\begin{equation}\label{lambda}
\lim_{\eta\to 0}\frac{\eta}{\lambda_\mathrm{rec}(\eta)}=0,
\end{equation}
and $|t_2-t_1| \leq \frac{1}{\lambda_\mathrm{rec}(\eta)}$.
\end{itemize}
Moreover, we say that the reconnection is {\em fast} if for any $a>0$ the function $\lambda_\mathrm{rec}$ satisfies
\begin{equation}\label{lambda fast}
\lim_{\eta\to 0}\frac{\eta^a}{\lambda_\mathrm{rec}(\eta)}=0.
\end{equation}
\end{defn}
Note that the rate of reconnection $\lambda_\mathrm{rec}(\eta)$ is essentially the inverse of the reconnection time, namely that time after which we are able to show that the topology of the solution has changed. 

We are now in position to state our first main result.
\begin{mainthm}[Accelerated Reconnection]\label{MainThm}
Let $r > 5/2$. There exists $\eta_0 >0$ sufficiently small such that there exists a family of divergence-free vector fields $(u^\mathrm{in}_\eta,b^\mathrm{in}_\eta)$ in $H^r(\T^3)$ with $0<\eta<\eta_0$, such that the unique local solution $(u_{\eta},b_{\eta},p_{\eta})$ of \eqref{eq:mhd} arising from $(u^\mathrm{in}_\eta,b^\mathrm{in}_\eta)$ shows accelerated magnetic reconnection with rate
$$
\lambda_\mathrm{rec}(\eta) := c_2 \frac{\eta^\frac12}{|\ln \eta|},
$$
for some positive constant $c_2>0$. More precisely, the magnetic field $b_{\eta}(t,\cdot)$ has an equilibrium point for times 
$t \in [0, c_1]$ and has no equilibrium points for times $t \in \left[c_2 \frac{|\ln \eta|}{\eta^\frac12}, 2 c_2 \frac{|\ln \eta|}{\eta^\frac12}\right]$, where the constants $c_1, c_2$ only depends on $r$. 
The statement is structurally stable, in the sense that it still holds 
for all solution with (divergence-free) initial data $(\tilde u^\mathrm{in}_\eta, \tilde b^\mathrm{in}_\eta)$
such that 
\begin{equation}\label{eqPertIntro}
\| \tilde u^\mathrm{in}_\eta - u^\mathrm{in}_\eta\|_{H^r} 
 + \|\tilde b^\mathrm{in}_\eta -  b^\mathrm{in}_\eta \|_{H^r} \leq c(r, \eta)
\end{equation}
with $c(r, \eta)$ sufficiently small.
\end{mainthm}

We now outline the main ideas of the construction and compare them with our previous result in \cite{CCL}. To simplify the notations, we drop the subscript $\eta$ in the solution/initial datum; we will return on this point later, see Remark \ref{remark dati iniziali}. We will construct $(u^\mathrm{in},b^\mathrm{in})$ such that:
\begin{itemize}
\item $b^\mathrm{in}$ has at least one hyperbolic zero $x^*$.  
\item there exists $T_{max} \gg \frac{|\ln \eta|}{\eta^{1/2}}$ and a unique solution $(u,b) \in C([0, T_{max}); H^{r}(\T^3))$ 
 to the \eqref{eq:mhd} with initial datum $(u^\mathrm{in},b^\mathrm{in})$;  
\item the solution satisfies $|b(t,x)|>0$ for any $x\in\T^3$ and $ t \in \left[c_2 \frac{|\ln \eta|}{\eta^\frac12}, 2 c_2 \frac{|\ln \eta|}{\eta^\frac12}\right]$.
\item In fact, the solution $b(t,x)$ will have one hyperbolic zero $x^*(t)$ for all $t \in [0,c_1]$, with $c_1$ independent on $\eta$ (structural stability with respect to time).  
\end{itemize}
The loss of equilibrium points implies that the topology of the magnetic lines has changed. Moreover, we point out that the zeros of the magnetic field are commonly observed as reconnection sites, see \cite{Priest, Schindler}.
We recall that hyperbolic zeros are stable under smooth perturbations, which means that the solution $b(t,\cdot)$ will still have at least one zero for all times $t \in [0,c_{\eta}]$ where $c_{\eta}$ is a sufficiently small constant. However, we will be able to show the remarkable fact that this constant $c_{\eta} = c_1$ is actually~$\eta$-independent,
which ensures the observability of the reconnection for arbitrarily small values of the resistivity $\eta$. 

Once we have proved that the vector field $b(t^*,\cdot)$ has no zeros for some $t^* > c_1$, the reconnection must have occurred at some intermediate time. Providing a quantitative bound for $t^*$ in terms of $\eta$ will be then the core of the proof: to show that the reconnection occurs in a regime which is faster than the resistive one, we will prove that for $0 < \eta \ll 1 $, we can take $t^* =\frac{1}{\lambda_\mathrm{rec}(\eta)}$ with $\lambda_\mathrm{rec}(\eta)$ that satisfies \eqref{lambda}, namely for a reconnection time scale that is much smaller than the linear diffusive one. In fact we can choose $t^*$ in the interval
$\left[c_2 \frac{|\ln \eta|}{\eta^\frac12}, 2 c_2 \frac{|\ln \eta|}{\eta^\frac12}\right]$.

To achieve this, the choice of the initial velocity field $u^\mathrm{in}$ will be crucial. The building blocks of our construction will be velocity fields that are {\em dissipation enhancing}, as detailed in Definition \ref{def:diss enh}. This is the key difference from our previous constructions in \cite{CCL}. In this latter, we do not assume any specific information on the initial velocity field and we simply set $u^\mathrm{in}=0$, performing a perturbative argument on vector fields that simultaneously solve the heat equation and the stationary Euler equations. For this type of solution, the advective term does not accelerate the reconnection process. In contrast, a velocity field $U$ is {\em dissipation enhancing} if the solutions to the corresponding linear advection-diffusion equation dissipate energy at a rate faster than the purely diffusive one. More precisely, we consider the equation
\begin{equation}\label{eq:ad-intro}\tag{AD}
\begin{cases}
\partial_t \rho+U\cdot\nabla \rho=\eta\Delta \rho,\\
\rho(0,\cdot)=\rho^\mathrm{in},
\end{cases}
\end{equation}
where $U:\T^2\to\R^2$ is an autonomous divergence-free vector field and $\rho^\mathrm{in}$ a mean-free initial datum. By using the incompressibility and Poincarè's inequality, a simple $L^2$ estimate gives that
\begin{equation}\label{decadimento calore}
\|\rho(t,\cdot)\|_{L^2}\leq e^{-C\eta t} \|\rho^\mathrm{in}\|_{L^2},
\end{equation}
for some constant $C>0$ depending on the constant in Poincarè's inequality. The inequality \eqref{decadimento calore} implies that the energy decays exponentially fast in time but it does not take into account effects coming from the advection term. Indeed, the same decay rate holds for the heat equation.
By defining the {\em dissipation time} $t_\mathrm{dis}$ to be the smallest $t>0$ such that
$$
\|\rho(t,\cdot)\|_{L^2}\leq \frac12\|\rho^\mathrm{in}\|_{L^2},
$$ 
the inequality \eqref{decadimento calore} implies $t_\mathrm{dis}\lesssim \eta^{-1}$. Thus, one can define enhanced dissipation as situations where $t_\mathrm{dis}\ll\eta^{-1}$. The enhanced effect of autonomous vector fields has been analyzed in different settings, we refer to \cite{BW, BCZ, BCZGH, BCZM, CKRZ, CZDE, CZD, CZDr, CLS, DZ, Dolce, DJS, FI, MZ, W, Z} and references therein.
More quantitatively, a vector field $U$ is dissipation enhancing with {\em dissipation rate }$\lambda(\eta)$ if
\begin{equation}\label{def:enhanced_intro}
\|\rho(t,\cdot)\|_{L^2} \lesssim e^{- \lambda(\eta)t} \|\rho^\mathrm{in}\|_{L^2}, \qquad \mbox{ with }\frac{\eta}{\lambda(\eta)}\to 0,\quad \mbox{ as }\eta\to 0,
\end{equation}
for every $\rho^\mathrm{in}\in L^2$ with zero streamline-average, with $\lambda(\eta)$ being a continuous positive increasing function. The restriction on the initial datum implies that $\rho^\mathrm{in}$ is orthogonal to the kernel of the transport operator $U \cdot \nabla$ and it ensures interaction with the mixing mechanism induced by the flow. However, this property is generally not preserved by the advection-diffusion equation \eqref{eq:ad-intro}, unless the vector field possesses particular symmetries. This is the case of shear flows or radial vector fields, where the streamline-average is invariant under the evolution. Therefore, while zero streamline-average at time zero is a necessary condition to expect enhanced dissipation, it is not always sufficient, and further structural assumptions on the flow may be required.\\

To obtain a quantitative bound for the reconnection time, the following theorem  is crucial.

\begin{mainthm2}[Enhanced dissipation of Sobolev norms]\label{teo:enhHr}
Let $U\in W^{S,\infty}(\T^d)$, for some integer $S\geq 1$ and $d\geq 2$, be a divergence-free vector field with dissipation enhancing rate $\lambda(\eta)$ and let $\rho^\mathrm{in}\in H^r(\T^d)$ for some integer $r\geq 0$ with zero streamlines-average. Let $\rho$ be the unique solution of \eqref{eq:ad-intro} with initial datum $\rho^\mathrm{in}$. 
There exists $\eta_0\in(0,1)$ such that the following holds for all 
$\eta \in (0, \eta_0)$.
For any $0\leq s\leq \min\{r,S\}$ and for all $\varepsilon >0$ 
\begin{equation}\label{eq:enhanced hr-intro}
\|\rho(t,\cdot)\|_{H^s}\lesssim_{\varepsilon,s} \frac{ 1+ \|U\|_{W^{S,\infty}}^s}{\eta^{s(\frac12 + \varepsilon)}} e^{- \lambda(\eta)t}\|\rho^\mathrm{in}\|_{H^s}.
\end{equation} 
The $\varepsilon$-loss can be removed if $U \in B^{S}_{\infty,2}(\T^d)$. In this case one has
\begin{equation}\label{eq:enhanced hrBesov-intro}
\|\rho(t,\cdot)\|_{H^s} \lesssim_{s} \frac{ 1+  \|U\|_{B^{S}_{\infty,2} }^s}{\eta^{s/2 }} e^{-  \lambda(\eta)t}\|\rho^\mathrm{in}\|_{H^s}.
\end{equation}
\end{mainthm2}

Concretely, the condition $B^{S}_{\infty,2}$ is, for instance, satisfied if $f \in W^{S', \infty}$ for some with $S' > S$, we refer to
to \eqref{Besov-Sobolev1}-\eqref{Besov-Sobolev2} for the definition of the Besov and Sobolev norms.

In Section \ref{Section:Proof} we will prove a more general version of this theorem, namely Theorem B', which holds also for non-autonomous vector fields such as those constructed in \cite{ACM, BBP, BBP2, BCZG, Coo, CooIS, CZNF, ELM, EZ, FI, MHSW22}, thus it provides an enhanced dissipation rate for high norms with a logarithmic $\lambda$, as a consequence of \cite[Theorem 2.5]{CZDE}, using exponential mixers.
We emphasize that Theorem B, and Theorem B' as well, is of independent interest, as it extends any quantitative enhanced dissipation estimates from the $L^2$ setting to higher Sobolev norms, with the (expected) drawback of having a constant which is unbounded in $\eta$ on the right hand side of \eqref{eq:enhanced hr-intro}. Specifically, the abstract framework in \cite{CZDE} requires working in a Hilbert space $H$ such that the solution $\rho^0$ of the inviscid problem
\begin{equation}\label{eq:te-intro}
\begin{cases}
\partial_t \rho^0+U\cdot\nabla \rho^0=0,\\
\rho^0(0,\cdot)=\rho^\mathrm{in},
\end{cases}
\end{equation}
satisfies the conservation
$$
\|\rho^0(t,\cdot)\|_{H}=\|\rho^\mathrm{in}\|_H,
$$
which does not hold for the choice $H=H^s$ with $s> 0$.
\\

The choice of the initial velocity field $u^\mathrm{in}$ in Theorem A, along with the quantitative rate in \eqref{eq:enhanced hr-intro}, provides an estimate for the reconnection time. Furthermore, the construction is highly flexible and can be adapted to different spatial domains, we will revisit this point at the end of Section \ref{Section:Proof}.  
It is important to note that the result is structurally stable in the following sense: the solution retains a hyperbolic equilibrium for all times $t \in [0, c_1]$, where the constant $c_1$ is independent on $\eta$. On the other hand, the solution has no equilibrium points for times $t \in \left[c_2 \frac{|\ln \eta|}{\eta^\frac12}, 2 c_2 \frac{|\ln \eta|}{\eta^\frac12}\right]$, see \eqref{FinalGoal}. Note that this time interval becomes arbitrarily large as $\eta \to 0^{+}$. In fact, it will be clear by the proof that, given any target time $T > c_2 \frac{|\ln \eta|}{\eta^\frac12} $, we can even choose the initial datum in such a way that the solution has no regular equilibrium points for all times $t \in \left[c_2 \frac{|\ln \eta|}{\eta^\frac12}, T\right]$. Furthermore, this topological picture remains unchanged for sufficiently small $H^r(\T^3)$-perturbations of the initial datum, the size of the perturbation being $\eta$-dependent (see \eqref{eqPertIntro}). 

We emphasize that we consider the inviscid model \eqref{eq:mhd} to simplify the construction of our building blocks, see Section \ref{Sec:perturbative}, but the proof can be generalized to the viscous model through a perturbative argument. 
We also remark that, while models based on Sweet-Parker layers or tearing instabilities exhibit reconnection on time scales faster than the purely diffusive regime, they are still considered slow. In contrast, physically {\em fast reconnection} involves mechanisms, such as Hall effects or kinetic processes, that yield rates that weakly depend on the resistivity, for example of logarithmic type.
Our approach can be extended to provide examples of accelerated reconnection for the Hall-MHD system as well. However, in our construction the Hall term does not contribute to the acceleration of the process, so we restrict our attention to \eqref{eq:mhd} in order to minimize technicalities.

Since our construction relies on steady solutions of the 2D Euler equations which are dissipation enhancing, the resulting reconnection rate cannot exceed polynomial scaling in the resistivity. Indeed, in \cite{BM} it has been proved that two-dimensional autonomous velocity fields cannot mix faster than $1/t$ and then the enhanced dissipation must also be polynomial, see \cite{CZDE}. Achieving faster rates would require non-autonomous fields. If we consider the viscous MHD model and we allow for the presence of a forcing term in the velocity equation, our strategy yields a method to construct solutions exhibiting reconnection on faster time scales, in particular of logarithmic type as predicted by the {\em Petschek's model}, see \cite{Priest}. Notice that a reconnection rate like $\lambda_\mathrm{rec}(\eta)=1/|\ln \eta|$ satisfies \eqref{lambda fast} in Definition \ref{def:fast rec}. To achieve this, we need to consider a {\em stochastic force}, relying on the results proved in \cite{BBP3, BBP2}. In particular, we consider the system
\begin{equation}\label{eq:mhd2-intro}\tag{F-MHD}
\begin{cases}
\partial_t u_{\eta}+(u_{\eta} \cdot \nabla)u_{\eta}+\nabla p_{\eta}= \Delta u_{\eta}+(b_{\eta}\cdot \nabla)b_{\eta}+f,\\
\partial_t b_{\eta}+(u_{\eta} \cdot \nabla)b_{\eta}=(b_{\eta} \cdot \nabla)u_{\eta}+\eta\Delta b_{\eta},\\
\dive u_{\eta}=\dive b_{\eta}=0,\\
u_{\eta}(0,\cdot)=u_{\eta}^\mathrm{in},\hspace{0.3cm} b_{\eta}(0,\cdot)=b_{\eta}^\mathrm{in},
\end{cases}
\end{equation}
where $f:[0,T]\times\T^3\to\T^3$ is a given force. The magneto-hydrodynamic equations driven by stochastic forces (in both the velocity and magnetic field equations) were introduced to model interactions between a conducting fluid and a magnetic field in the presence of random perturbations, see \cite{Landau}.
In order to construct solutions of \eqref{eq:mhd2-intro} exhibiting {\em fast reconnection}, the starting point will be velocity fields that are solutions to the two-dimensional Navier-Stokes equations with a stochastic force as those in \cite{BBP}, instead of stationary Euler flows. Then, with the same strategy of Theorem A we can prove the following, see Section \ref{sec:stocastica}.
\begin{mainthm4}[Fast reconnection]
Let $r > 5/2$. There exists $\eta_0 >0$ sufficiently small such that there exists a family of divergence-free vector fields $(u^\mathrm{in}_\eta,b^\mathrm{in}_\eta)$ in $H^r(\T^3)$ with $0<\eta<\eta_0$, and a stochastic force $f\in L^\infty((0,T);H^r(\T^3))$, defined on some given filtered probability space $(\Omega,\mathcal{F},\mathcal{F}_t,\mathbb{P})$, such that the following holds: for any $\delta>0$ there exists a subset $\bar \Omega\subset \Omega$ with $\mathbb{P}(\Omega\setminus\bar\Omega)\leq \delta$ such that the unique local solution $(u_{\eta},b_{\eta},p_{\eta})$ of \eqref{eq:mhd2} arising from $(u^\mathrm{in}_\eta,b^\mathrm{in}_\eta)$ shows fast magnetic reconnection with rate 
$$
\lambda_\mathrm{rec}(\eta) := \frac{c_2}{|\ln \eta|},
$$
for some positive deterministic constant $c_2>0$, and for all realizations $\omega\in\bar\Omega$. More precisely, the magnetic field $b_{\eta}(t,\cdot)$ has an equilibrium point for times $t \in [0, c_1]$ and has no equilibrium points for times 
$$
t \in \left[c_2 |\ln \eta|, 2 c_2|\ln \eta|\right].
$$  
The statement is structurally stable, in the sense that it still holds 
for all solution with (divergence-free) initial data $(\tilde u^\mathrm{in}_\eta, \tilde b^\mathrm{in}_\eta)$
such that 
\begin{equation}\label{eqPertIntro}
\| \tilde u^\mathrm{in}_\eta - u^\mathrm{in}_\eta\|_{H^r} 
 + \|\tilde b^\mathrm{in}_\eta -  b^\mathrm{in}_\eta \|_{H^r} \leq c(r, \eta)
\end{equation}
with $c(r, \eta)$ sufficiently small.
\end{mainthm4}

It is worth mentioning that Alfven's Theorem is still true for the non-resistive ($\eta = 0$) {\em stochastic viscous} model \eqref{eq:mhd2-intro}. We conclude this introduction by pointing out recent works on {\em vortex reconnection} for the three-dimensional Navier-Stokes equations \cite{CL, ELP}. Magnetic reconnection and vortex reconnection share fundamental similarities in their underlying physical principles, as both involve the interaction and reconfiguration of flow structures due to topological changes.  In vortex reconnection, vortices (regions of rotating fluid) interact, merge, or split, leading to a redistribution of vorticity and changes in the flow structure. Both phenomena are highly nonlinear and can exhibit complex, chaotic behavior, often involving multiple scales of interaction. It would be interesting to produce examples of solutions of the three-dimensional Navier-Stokes equations that reconnect more rapidly, in the spirit of this work. We also point out the recent work \cite{Boss}, in which stationary solutions of the 3D Euler system with H\"older regularity are constructed (using convex integration techniques) that share the same topology of a given smooth divergence-free vector field.
 
Finally, regarding topological results for MHD equations, we also mention the recent work \cite{EP-top}, where a sophisticated counting argument was employed to prove obstructions to topological relaxation.

\subsection*{Outline of the paper.} In Section \ref{Sec:preliminari}, we introduce the notation and provide some background results. In Section \ref{Sec:perturbative}, we construct explicit solutions to the \eqref{eq:mhd} equations and we prove a stability result for strong solutions by appropriately quantifying the time of existence. This will ensure that the solutions we construct exist over a time interval longer than the reconnection time. Finally, in Section \ref{Section:Proof}, we first prove the enhanced dissipation effect for high Sobolev norms in Theorem B', and then Theorem A. Finally, in Section \ref{sec:stocastica} we prove the fast reconnection result for the forced viscous system.

\section{Preliminaries}\label{Sec:preliminari}
In this section we fix the notations and we recall some preliminary results.

\subsection{Tools from Harmonic Analysis}
For any $s\in\R$ we define the inhomogeneous symbol $\langle D\rangle ^s:=(\mathrm{Id}-\Delta)^{s/2}$ and the homogeneous symbol $D^s:=(-\Delta)^\frac{s}{2}$. We recall that $D^s$ and $\langle D\rangle^s$ are Fourier multiplier operators on $\T^d$ with symbols $(1 + 4\pi^2|k|^2)^{2s}$ and $(2\pi|k|)^s$, respectively. This means that
$$
\langle D\rangle ^s f(x):=\sum_{k\in\Z^d}(1 + 4\pi^2|k|^2)^{2s}\hat{f}(k)e^{i2\pi k\cdot x}, \qquad D^s f(x):=\sum_{k\in\Z^d}(2\pi|k|)^s\hat{f}(k)e^{i2\pi k\cdot x}.
$$
For any $s\in\R$ we define the Sobolev space $H^s(\T^3)$ and the homogeneous Sobolev space $\dot{H}^s(\T^3)$ as the closure of the Schwartz functions under their respective norm (or semi-norm in the case of $\dot{H}^s$)
$$
\|f\|_{H^s}:=\|\langle D\rangle ^s f\|_{L^2},\qquad \|f\|_{\dot{H}^s}:=\|D^s f\|_{L^2}.
$$
We recall the classical Kato-Ponce commutator estimate, see \cite{Kato Ponce}.
\begin{lem}\label{lem:kato ponce}
Let $s>0$ and $1<p<\infty$. Then, for any $f,g\in C^{\infty}(\T^3)$
\begin{equation}
\|\langle D \rangle^s(fg)-f\langle D \rangle^s(g)\|_{L^p}\leq C\|\nabla f\|_\infty\|\langle D \rangle^{s-1}g\|_{L^p}+\|\langle D \rangle^{s}f\|_{L^p}\| g\|_{L^\infty}.
\end{equation}
\end{lem}

The following inequality is known as {\em fractional Leibniz rule}, see \cite{Grafakos}.
\begin{lem}\label{lem:fractional}
Let $s>0$, $1<r<\infty$ and $1<p_1,q_1,p_2,q_2\leq \infty$ such that
$$
\frac1r=\frac{1}{p_1}+\frac{1}{q_1}=\frac{1}{p_2}+\frac{1}{q_2}.
$$
Then, there exists a constant $C>0$ (depending on $s,r,p_1,q_1,p_2,q_2$) such that for any $f,g\in C^{\infty}(\T^3)$\begin{align}
\|D^s(fg)\|_{L^r}\leq C\left(\|f\|_{L^{p_1}}\|D^s g\|_{L^{q_1}}+\|D^sf\|_{L^{p_2}}\| g\|_{L^{q_2}}\right),\\
\|\langle D\rangle^s(fg)\|_{L^r}\leq C\left(\|f\|_{L^{p_1}}\|\langle D\rangle^s g\|_{L^{q_1}}+\|\langle D\rangle^sf\|_{L^{p_2}}\| g\|_{L^{q_2}}\right).
\end{align}
\end{lem}

We now recall how to adapt the classical Littlewood-Paley decomposition to periodic functions. We use the notations of \cite{ChG09, Tao book}.
Let $\varphi\in \mathcal{S}(\R^3)$ be a Schwartz function that the Fourier transform $\widehat{\varphi}$ satisfies 
\begin{equation}
\widehat{\varphi}(\xi)=\begin{cases}
1\qquad \mbox{for }|\xi|\leq 1,\\
0 \qquad \mbox{for }|\xi|>2.
\end{cases}
\end{equation}
For $N >0$ we define the function $\varphi_N(x):=N^{3}\varphi(N x)$, and  
for~$f \in L^2(\T^3)$:
\begin{equation}\label{def:Sj}
P_{\leq N} f(x) := \sum_{k \in \Z^3} \widehat{\varphi_N}(k) \, \widehat{f}(k)\, e^{i k \cdot x} 
\end{equation}
and
\begin{equation}
P_N := P_{\leq N} - P_{\leq N/2}. \label{def:deltaj}
\end{equation}
We denote with $\N$ the set of non-negative integers and $\N_*:=\N\setminus \{0\}$.
We have the identity (in the $L^2$ sense, or point-wise if $f \in H^{s}(\T^3)$ with $s > 3/2$):
$$
f = P_{\leq 1} + \sum_{N \in 2^{\N_*}} P_N f .
$$
Note that the Fourier coefficients of $P_N f$ are zero outside the annulus 
$$
\Big\{ k \in\R^3: \frac12 N<|\xi|<2N\Big\},
$$
and the Fourier coefficients of $P_{\leq 1} f$ are zero outside the ball $\{\xi\in\R^3: |\xi|<2 \}$.\\
On the other hand we also have  (this is the main identity usually used to prove the Poisson summation formula)
$$
\sum_{k \in \Z^3} \varphi_N(x+k)  = \sum_{k \in \Z^3} e^{i k\cdot x} \, \widehat{\varphi_N}(k). 
$$
Thus letting 
$$
\Phi_N(x) := \sum_{k \in \Z^3} \varphi_N(x+k),
$$
we can rewrite
$$
P_{\leq N} f = \Phi_N * f. 
$$
Since 
$$
\| \Phi_N \|_{L^{p}(\T^3)} \simeq \| \varphi_N \|_{L^{p}(\R^3)} \simeq N^{3(1 - \frac1p)}, \qquad N \geq 1,
$$ 
we can recover the classical Bernstein inequality in the periodic setting, as long as $N \geq 1$. We also define
$$P_{> N} := \mathrm{Id} - P_{\leq N}$$ 
and we recall some useful telescoping identities
$$
P_{\leq N}f=\sum_{M\leq N}P_M f,\qquad P_{> N}f=\sum_{M>N}P_M f, \qquad N, M \in 2^{\N}.
$$
We are now in position to state the Bernstein's inequalities, see \cite{Tao book}.
\begin{lem}
Let $N \geq 1$. Then, for any $s\geq0$ and $1\leq p\leq q\leq\infty$, the following inequalities hold.
\begin{align}
\|P_{\geq N} f\|_{L^p}&\lesssim_{s,p} C N^{-s}\|D^sP_{\geq N} f\|_{L^p},\label{eq:bern1}\\
\|D^sP_{\leq N} f\|_{L^p}&\lesssim_{s,p} C N^s\|P_{\leq N} f\|_{L^p},\\
\|P_{\leq N} f \|_{L^q}&\lesssim_{p,q}N^{\frac3p-\frac3q}\|P_{\leq N} f \|_{L^p}.
\end{align}
Moreover, if $N \in 2^{\N_*}$ then
\begin{align}
\|P_N D^{\pm s}f\|_{L^p}&\sim_{p,s} N^{\pm s}\|P_N f\|_{L^p},\\
\|P_N f \|_{L^q}&\lesssim_{p,q}N^{\frac3p-\frac3q}\|P_N f \|_{L^p}.\label{eq:bern5}
\end{align}
\end{lem}
We also recall that, by Plancharel's Theorem, the following holds 
\begin{equation}\label{definizione besov}
\|f\|_{\dot{H}^s}\sim_{p,d}\big(\sum_N N^{2s}\|P_N f\|_{L^2}^2\big)^\frac12,\qquad \|f\|_{H^s}\sim_{p,d}\|P_{\leq 1}f\|_{L^2}+ \big(\sum_{N>1} N^{2s}\|P_N f\|_{L^2}^2\big)^\frac12.
\end{equation}
Furthermore, a central feature of this decomposition is given by the Littlewood-Paley inequality
\begin{equation}
\|f\|_{L^p}\sim_{p,d}\|\big(\sum_N |P_N f|^2\big)^\frac12\|_{L^p},
\end{equation}
which holds for all $1<p<\infty$, see \cite[Theorem 6.1.2]{Grafakos}.\\

One can use the Littlewood-Paley projections also as a tool to define Besov-Spaces, as follows 

\begin{equation}\label{Besov-Sobolev1}
\| f \|_{B^{s}_{p,q}}^q 
 :=  \| P_{\leq 1} f\|_{L^{p}}^q + \sum_{N  \in 2^{\N_*}} N^s \| P_N f\|_{L^{p}}^q, \qquad
\| f \|_{B^{s}_{p,\infty}} 
 :=  \| P_{\leq 1} f\|_{L^{p}} + \sup_{N  \in 2^{\N_*}} N^s \| P_N f\|_{L^{p}}.
\end{equation}
We remark that the $W^{s, \infty}$-Sobolev norms is defined by interpolation from 
\begin{equation}\label{Besov-Sobolev2}
\|f\|_{W^{s, \infty}} := \sup_{\alpha : |\alpha| \leq s} \| \partial^{\alpha} f \|_{L^\infty}, \alpha \in \N^{3}, s \in \N,
\end{equation}
satisfies $\|f\|_{W^{s, \infty}} \simeq_{s} \| f \|_{B^{s}_{\infty,\infty}}$.

We now recall the following result, see \cite[Lemma 2.97]{BCD}.
\begin{lem}[Commutator estimate 1]
There exists a constant $C>0$ such that for any Lipschitz function $f$ with gradient in $L^p$ and any function $g\in L^q$, we have
\begin{equation}
\|[P_N,f]g\|_{L^r}\leq CN^{-1}\|\nabla f\|_{L^p}\|g\|_{L^q},\qquad \mbox{ with }\,\,\frac{1}{p}+\frac{1}{q}=\frac{1}{r},
\end{equation}
for any dyadic number $N\in 2^{\N}$.
\end{lem}
We remark again that the action of the projection operator $P_N$ in the physical space corresponds to a convolution with a smooth mollifier. The following commutator estimate also applies, which in a sense is reminiscent of the DiPerna-Lions' commutator estimate \cite{DPL}.
\begin{lem}[Commutator estimate 2]\label{lem:commutatore}
There exists a constant $C>0$ such that for any Lipschitz function $f$ with gradient in $L^p$ and any function $g\in L^q$, we have
\begin{equation}
\|[P_N,f\cdot\nabla]g\|_{L^r}\leq C\|\nabla f\|_{L^p}\|g\|_{L^q},\qquad \mbox{ with }\,\,\frac{1}{p}+\frac{1}{q}=\frac{1}{r},
\end{equation}
for any dyadic number $N\in 2^{\N}$.
\end{lem}

\subsection{Enhanced dissipation}
We consider the Cauchy problem for the linear advection-diffusion equation
\begin{equation}\label{eq:ad}\tag{AD}
\begin{cases}
\partial_t\rho+U\cdot\nabla\rho=\eta\Delta\rho,\\
\rho(0,\cdot)=\rho^\mathrm{in},
\end{cases}
\end{equation}
where $U : \T^2 \to \R^2$ is a given autonomous divergence-free vector field, $\rho^\mathrm{in}:\T^2\to\R$ is a given mean-free initial datum and $\eta>0$ is a diffusion parameter. Notice that the average of the solution is preserved by the equation, so that there is no loss of generality in the mean-free assumption. We also assume that the velocity field and the initial datum are smooth. Under these circumstances it is well-known that the equation \eqref{eq:ad} admits a unique smooth solution.\\

We now provide a precise quantitative definition of enhanced dissipation, as in \cite{BCZM}.
\begin{defn}\label{def:diss enh}
Let $\eta_0\in(0,1)$ and $\lambda: (0,\eta_0)\to(0,+\infty)$ be a continuous increasing function such that
\begin{equation}
\lim_{\eta\to 0}\frac{\eta}{\lambda(\eta)}=0.
\end{equation}
The velocity field $U$ is {\em dissipation enhancing} at rate $\lambda(\eta)$ if there exists $C\geq 1$ only depending on $U$ such that if $\eta\in(0,\eta_0)$ then for every $\rho^\mathrm{in}\in L^2(\T^2)$ with zero streamlines-average we have the enhanced dissipation estimate
\begin{equation}\label{def:diff enhanced}
\|\rho(t,\cdot)\|_{L^2} \leq C e^{-\lambda(\eta)t} \|\rho^\mathrm{in}\|_{L^2},
\end{equation}
for every $t\geq 0$.
\end{defn}
In \cite{CKRZ} it has been shown that the enhanced dissipation property is equivalent to the non-existence of non-trivial $H^1$-eigenfunctions of the operator $U\cdot\nabla$. In particular, functions that are constant on streamlines must be excluded since they are eigenfunctions. To be more precise, a function $\rho^\mathrm{in}\in L^2(\T^2)$ has zero streamlines-average if
$$
\fint_\gamma \rho^\mathrm{in}\de s=0,
$$
for all integral curves $\gamma$ of $U$. This condition implies that $\rho^\mathrm{in}$ is not constant on the streamlines of $U$ and thus it does not belong to the kernel of the operator $U\cdot\nabla$. As already pointed out in the introduction, this property is not in general preserved under the evolution given by the equation \eqref{eq:ad}, unless one does not assume that $U$ posses particular symmetries. This is the case, for example, of shear flows. We recall a result from \cite{BCZ} which will be used in the proof of Theorem A.\\

Let $U(x_1,x_2)=(f(x_2),0)$ be a shear flow and define the functional space
\begin{equation}\label{def:hshear}
H_\mathrm{shear}:=\left\{\rho\in L^2(\T^2): \int_\T \rho(x_1,x_2)\de x_1=0,\,\,\mbox{ for a.e. }x_2\in \T\right\}.
\end{equation}
The vector field $U$ is a stationary solution of the 2D Euler equations. Assume that $f\in C^{n_0+1}(\T)$ has a finite number of critical points, denoted by $\bar y_1,...,\bar y_N$ and where $n_0\in\N_*$ denotes the maximal order of vanishing of $f'$ at the critical points, namely, the minimal integer such that
$$
f^{(n_0)}(\bar y_i)\neq0,\qquad \mbox{for all }i=1,...,N.
$$
We further assume that $\int_\T f(y) \de y=0$. We will use a result proved in \cite{BCZ}, by following the formulation given in \cite{CZDr, CZG}.
\begin{thm}\label{teo:bcz}
There exists a positive constant $C>0$ depending on $f$ such that the enhanced diffusion rate on $H_\mathrm{shear}$ is
\begin{equation}
\lambda(\eta)=C\,\eta^\frac{n_0}{n_0+2}.
\end{equation}
\end{thm}

\section{Perturbative analysis}\label{Sec:perturbative}
The goal of this section is twofold: first, in Subsection \ref{Sec:Bifurcation} we provide examples of explicit global-in-time strong solutions of the \eqref{eq:mhd} equations that depend only on the first two variables. Second, in Subsection \ref{Sec:Stability}, we establish the existence of strong solutions emanating from data close to a given one, with a quantitative control on the lifespan in terms of the initial perturbation. Although the stability itself is classical, the quantitative estimate will be essential for the proof of Theorem A: perturbing the solutions constructed in Subsection \ref{Sec:Bifurcation}, we ensure that the resulting solutions are defined long enough for magnetic reconnection to occur.

\subsection{Construction of a reference solution}\label{Sec:Bifurcation}
Let $d\geq 2$ be a positive integer, we say that a point $x_0\in\T^d$ is a {\em zero} or an {\em equilibrium point} of a vector field $v\in C^0(\T^d)$ if $v(x_0)=0$. A zero $x_0$ of a vector field $v\in C^1(\T^d)$ is {\em non-degenerate} if $\nabla v(x_0)$ is an invertible matrix. A non-degenerate zero is {\em hyperbolic} if $\nabla v(x_0)$ has no eigenvalues with zero real part. We recall that a vector field is structurally stable in a neighborhood of a hyperbolic zero. It is easy to show, using the implicit function theorem, that hyperbolic zeros are preserved by $C^1$ perturbations. \\

We now construct special solutions of the three-dimensional \eqref{eq:mhd} equations depending only on the first two variables $x_1, x_2$. We consider solutions of the form
\begin{equation}\label{DefUB}
\tilde u := ( U_1, U_2, 0), 
\qquad \tilde b := (0,0, \tilde b_3),
\end{equation}
where $U : (x_1, x_2) \in \T^2 \to \R^2$ is a stationary solution of the 2D Euler equations, i.e. it satisfies
\begin{equation}\label{TaylorDef}
\dive U  = 0, \qquad (U \cdot \nabla ) U = - \frac12 \nabla (|U|^{2}),
\end{equation}
and the scalar function $\tilde b_3 : (t,x_1, x_2) \in [0, \infty)\times\T^2 \to \R$ is a solution of the advection-diffusion equation
\begin{equation}\label{AdvDiffForB3}
\begin{cases}
\partial_t \tilde b_3   + (U \cdot \nabla ) \tilde b_3 = \eta \Delta\tilde b_3,\\
\tilde b_3(0,\cdot)=\tilde b_3^\mathrm{in},
\end{cases} 
\end{equation}
with $\tilde b_3^\mathrm{in}:\T^2\to\R$ being a given smooth initial datum.
Note that \eqref{AdvDiffForB3} is a linear equation. It is immediate to check that the couple $(\tilde u,\tilde b)$ defined as in \eqref{DefUB} solves \eqref{eq:mhd} with pressure $P = - \frac12  |U|^2$ and initial velocity $\tilde u^\mathrm{in}=U$.\\
\\
We now impose additional conditions on the initial magnetic field to obtain the reconnection result. We will consider the third component $\tilde b^\mathrm{in}_3 : \T^2 \to \R$ of the magnetic field at time $t=0$ to be a smooth function which verifies:
\begin{itemize}
\item the average of $\tilde b^\mathrm{in}_3$ is positive, which means that 
$$ 
\langle \tilde b^\mathrm{in}_3 \rangle:= \int_{\T^2} \tilde b^\mathrm{in}_3(x_1,x_2) \de x_1 \de x_2>0,
$$
\item there exists some $(x_1^*,x_2^*)\in\T^2$ such that $\tilde b^\mathrm{in}_3(x_1^*,x_2^*)=0$. In particular, this implies that $(x_1^*,x_2^*)$ is an equilibrium point for $\tilde b$.  
\end{itemize}
Since the average of $\tilde b^\mathrm{in}_3$ is conserved by the equation \eqref{AdvDiffForB3} we have that $\langle \tilde b^\mathrm{in}_3 \rangle=\langle \tilde b_3(t) \rangle$ for all $t>0$.
The assumptions on $\tilde b^\mathrm{in}_3$ implies that $\tilde b^\mathrm{in}$ has a line of equilibrium points 
$$
\ell:=\{ (x_1^*,x_2^*,x_3)\in \T^2\times \T : x_3 \in \T \}.
$$ 
In particular, this implies that the equilibrium points $(x_1^*,x_2^*,x_3)$ of $\tilde b^\mathrm{in}$ are all degenerate, in the sense that 
$$
\det ( \nabla \tilde b^\mathrm{in}) \big|_{(x_1^*,x_2^*,x_3)} = 0.
$$ 
We will appropriately perturb these solutions to construct stable reconnection scenarios, while quantifying the corresponding reconnection time.

\subsection{Stability estimates}\label{Sec:Stability}
We now consider the following problem: given a smooth solution $(\tilde u,\tilde b,\tilde p)$ of \eqref{eq:mhd} defined on some time interval $[0,T]$ and starting from some initial data $(\tilde u^\mathrm{in},\tilde b^\mathrm{in})$, we want to quantify the time of existence of a solution starting from an $H^r$-initial data $(u^\mathrm{in},b^\mathrm{in})$ in terms of the difference $\|\tilde u^\mathrm{in}-u^\mathrm{in}\|_{H^r}+\|\tilde b^\mathrm{in}-b^\mathrm{in}\|_{H^r}$. The result is the following.
\begin{thm}\label{thm:stab}
Let $\tilde{u},\tilde{b}\in C([0,T^*];H^{r+1}(\T^3))$ with $r > \frac{5}{2}$ be a local strong solution of $\eqref{eq:mhd}$ starting from a divergence-free initial datum~$(\tilde{u}^\mathrm{in},\tilde{b}^\mathrm{in})$.  
Let $u^\mathrm{in},b^\mathrm{in}\in H^r(\T^3)$ be divergence-free vector fields and define the function
\begin{equation}\label{fdjklsdjnfgjhksdjbghjdksbg}
f(t):=\frac{1}{\sqrt{\|\tilde{u}^\mathrm{in}-u^\mathrm{in}\|^2_{H^r}+\|\tilde{b}^\mathrm{in}-b^\mathrm{in}\|^2_{H^r} }} - \frac{C}{2} \int_0^t e^{\frac{C}{2} \int_0^s \|  \tilde{u}(\tau, \cdot) \|_{H^{r+1}} + \|  \tilde{b} (\tau, \cdot) \|_{H^{r+1}}   \, \de \tau  } \, \de s,
\end{equation}
for some constant $C>0$.
Moreover, define $T:=\sup\{t\in\R: f(t)\geq 0\}$, with $T= \infty$ if $f(t)>0$ for all $t >0$. 
If $T^*\geq T$, then there exists a unique strong solution of \eqref{eq:mhd} with initial datum $(u^\mathrm{in},b^\mathrm{in})$ and with life-span~$[0, T)$. Furthermore, this solution verifies the following a-priori estimate 
\begin{equation}
\| \tilde{u} (t,\cdot)  - u (t,\cdot)  \|_{H^r}^2 + \| \tilde{b} (t,\cdot) - b (t,\cdot)   \|_{H^r}^2\leq \frac{e^{C\int_0^t   \|  \tilde{u}(s, \cdot) \|_{H^{r+1}} + \|  \tilde{b} (s, \cdot) \|_{H^{r+1}}  \,\de s}}{f(t)^2},
\end{equation}
for all $t\in[0,T)$.
\end{thm}

\begin{proof}
We denote by $P_{\tilde{u},\tilde{b}}$ and $P_{u,b}$ the pressure function of, respectively, $(\tilde{u}, \tilde{b})$ and $(u, b)$. Our aim is to construct, using a perturbative argument, a solution $(u,b)$ that remains close to $(\tilde{u}, \tilde{b})$ for at least a short period of time. Thus, we set 
\begin{equation}\label{Def:U}
u =  \tilde{u} + v,  \qquad b =  \tilde{b} + m,
\end{equation} 
where $(v, m)$ is the solution of the following system
\begin{equation}\label{eq:mhdPerturbative}
\begin{cases}
\partial_t v + (\tilde{u} \cdot \nabla)v  + (v\cdot \nabla)\tilde{u} + (v\cdot \nabla)v+\nabla P_v=
(\tilde{b} \cdot \nabla)m + ( m \cdot \nabla)\tilde{b} + (m \cdot \nabla)m,\\
\partial_t m + (\tilde{u} \cdot \nabla) m + (v \cdot \nabla) \tilde{b} + (v \cdot \nabla) m   = 
 (\tilde{b} \cdot \nabla) v + (m \cdot \nabla) \tilde{u} + (m \cdot \nabla) v + \eta \Delta{m},\\
\dive v=\dive m=0,\\
v(0,\cdot)=v^\mathrm{in},\hspace{0.3cm} m(0,\cdot)=m^\mathrm{in},
\end{cases}
\end{equation} 
with pressure $P_v:=P_{u,b} - P_{\tilde{u},\tilde{b}}$, and initial data defined as
$$
v^\mathrm{in}:=u^\mathrm{in}-\tilde{u}^\mathrm{in},\qquad m^\mathrm{in}:=b^\mathrm{in}-\tilde{b}^\mathrm{in}.
$$
We define the energies $e_r(t)$ as follows
\begin{equation*}
e_r(t):= \| v (t,\cdot) \|_{H^r}^2 + \| m (t,\cdot) \|_{H^r}^2.
\end{equation*}
A standard PDE argument guarantees that the solution $(u,b)$ exists as long as we are able to control the energies $e_r$. To this purpose, we apply the operator $\langle D \rangle^r$ to the equations obtaining that
\begin{align}
\partial_t \langle D \rangle^{r} v  &+  \nabla \langle D \rangle^{r} P_v + 
\langle D \rangle^{r} \big( ( v\cdot\nabla)v \big) + \langle D \rangle^{r} \big( ( \tilde{u}\cdot\nabla) v \big)
+\langle D \rangle^{r} \big( ( v\cdot\nabla) \tilde{u} \big)\nonumber\\
&= \langle D \rangle^{r} \big(( m \cdot \nabla) m \big)  + \langle D \rangle^{r} \big( ( \tilde{b}\cdot\nabla)m \big) + \langle D \rangle^{r} \big( ( m\cdot\nabla)\tilde{b}\big),
\end{align}
while for the magnetic-field equation
\begin{align}
\partial_t \langle D \rangle^{r} m + \eta \langle D \rangle^{r+2} m & +  \langle D \rangle^{r} \big ((v\cdot\nabla)m \big) +
\langle D \rangle^{r} \big ( ( \tilde{u}\cdot\nabla)m \big) + \langle D \rangle^{r} \big ( ( v\cdot\nabla)\tilde{b} \big)\nonumber\\
& = \langle D \rangle^{r} \big ( ( m\cdot\nabla)v \big) + \langle D \rangle^{r} \big ( ( m\cdot\nabla) \tilde{u} \big) 
+ \langle D \rangle^{r} \big ( (\tilde{b}\cdot\nabla) v \big).
\end{align}

We take the scalar product of these equations against $\langle D \rangle^{r} v$ and $\langle D \rangle^{r} m$ , respectively, sum them, and integrate over space. By using the divergence-free condition and applying integration by parts, we obtain
\begin{align}\label{fmkdlknfdhsjhbfhdjskhdfgn}
\frac12 & \frac{\de}{\de t} e_r(t) + \eta \int_{\T^3} | \langle D \rangle^{r + 1} m|^2 \, \de x = 
- \int_{\T^3} \langle D \rangle^{r} v \cdot \big[ \langle D \rangle^{r}, v \cdot \nabla \big]  v
- \int_{\T^3} \langle D \rangle^{r} v \cdot \big[ \langle D \rangle^{r}, \tilde{u} \cdot \nabla \big]  v
\\ \nonumber
&
- \int_{\T^3} \langle D \rangle^{r} v \cdot   \langle D \rangle^{r} \big( ( v\cdot\nabla) \tilde{u} \big)
+  \int_{\T^3} \langle D \rangle^{r} v \cdot \big[ \langle D \rangle^{r}, m \cdot \nabla \big]  m
+  \int_{\T^3} \langle D \rangle^{r} v \cdot \big[ \langle D \rangle^{r}, \tilde{b} \cdot \nabla \big]  m
\\ \nonumber
&
+ \int_{\T^3} \langle D \rangle^{r} v \cdot   \langle D \rangle^{r} \big( ( m\cdot\nabla) \tilde{b} \big)
-  \int_{\T^3} \langle D \rangle^{r} m \cdot \big[ \langle D \rangle^{r}, v \cdot \nabla \big]  m
-  \int_{\T^3} \langle D \rangle^{r} m \cdot \big[ \langle D \rangle^{r}, \tilde{u} \cdot \nabla \big]  m
\\ \nonumber
&
- \int_{\T^3} \langle D \rangle^{r} m \cdot   \langle D \rangle^{r} \big( ( v\cdot\nabla) \tilde{b} \big)
+  \int_{\T^3} \langle D \rangle^{r} m \cdot   \langle D \rangle^{r} \big( ( m \cdot\nabla) \tilde{u} \big).
\end{align}
We use the Kato-Ponce commutator estimate in Lemma \ref{lem:kato ponce}, the fractional Leibniz rule in Lemma \ref{lem:fractional} and the Sobolev embedding $H^{r}(\T^3) \hookrightarrow L^\infty(\T^3)$ for $r > \frac32$, to get the bounds
\begin{itemize}
\item $\| \big[ \langle D \rangle^{r}, v \cdot \nabla \big]  v \|_{L^2} \lesssim  \| v \|_{H^{r}} \| \nabla v \|_{L^{\infty}}\lesssim \| v \|_{H^{r}}^2$,

\item $\| \big[ \langle D \rangle^{r}, \tilde{u} \cdot \nabla \big]  v \|_{L^2} 
\lesssim  \| \tilde{u} \|_{H^r} \| \nabla v \|_{L^{\infty}} + \| \nabla \tilde u\|_{L^{\infty}} \| v \|_{H^r} \lesssim \|  \tilde{u} \|_{H^r} \| v \|_{H^r}$,

\item $\| \langle D \rangle^{r} \big( ( v\cdot\nabla) \tilde{u} \big) \|_{L^2} \lesssim \| v \|_{H^r} \| \nabla \tilde{u} \|_{L^{\infty}} + \|  v \|_{L^{\infty}} \| \nabla \tilde{u} \|_{H^r} \lesssim \|  \tilde{u} \|_{H^{r+1}} \| v \|_{H^r}$,

\item $\| \big[ \langle D \rangle^{r}, m \cdot \nabla \big]  m \|_{L^2} \lesssim  \| m \|_{H^{r}} \| \nabla m \|_{L^{\infty}}\lesssim \| m \|_{H^{r}}^2$,

\item $\| \big[ \langle D \rangle^{r}, \tilde{b} \cdot \nabla \big]  m \|_{L^2} \lesssim  \| \tilde{b} \|_{H^r} \| \nabla m \|_{L^{\infty}} + \| \nabla \tilde{b} \|_{L^{\infty}} \| m \|_{H^r} \lesssim \|  \tilde{b} \|_{H^r} \| m \|_{H^r}$,

\item $\| \langle D \rangle^{r} \big( ( m\cdot\nabla) \tilde{b} \big)\|_{L^2} \lesssim  \| m \|_{H^r} \| \nabla \tilde{b} \|_{L^{\infty}} + \|  m \|_{L^{\infty}} \| \nabla \tilde{b} \|_{H^r} \lesssim \|  \tilde{b} \|_{H^{r+1}} \| m \|_{H^r}$,

\item $\| \big[ \langle D \rangle^{r}, v \cdot \nabla \big]  m\|_{L^2} \lesssim  \| v \|_{H^r} \| \nabla m \|_{L^{\infty}} + \| \nabla v \|_{L^{\infty}} \| m \|_{H^r} \lesssim \|  v \|_{H^r} \| m \|_{H^r}$,

\item $\| \big[ \langle D \rangle^{r}, \tilde{u} \cdot \nabla \big]  m \|_{L^2} \lesssim  \| \tilde{u} \|_{H^r} \| \nabla m \|_{L^{\infty}} + \| \nabla \tilde{u} \|_{L^{\infty}} \| m \|_{H^r} \lesssim \|  \tilde{u} \|_{H^r} \| m \|_{H^r}$,

\item $\| \langle D \rangle^{r} \big( ( v\cdot\nabla) \tilde{b} \big)\|_{L^2} \lesssim  \| v \|_{H^r} \| \nabla \tilde{b} \|_{L^{\infty}} + \|  v \|_{L^{\infty}} \| \nabla \tilde{b} \|_{H^r} \lesssim \|  \tilde{b} \|_{H^{r+1}} \| v \|_{H^r}$,

\item $\| \langle D \rangle^{r} \big( ( m\cdot\nabla) \tilde{u} \big)\|_{L^2} \lesssim  \| m \|_{H^r} \| \nabla \tilde{u} \|_{L^{\infty}} + \|  m \|_{L^{\infty}} \| \nabla \tilde{u} \|_{H^r} \lesssim \|  \tilde{u} \|_{H^{r+1}} \| m \|_{H^r}$.
\end{itemize}
Plugging these estimates in \eqref{fmkdlknfdhsjhbfhdjskhdfgn}, using H\"older and Young inequalities we arrive to
$$
\frac{\de}{\de t} e_r(t) \leq C  (e_r(t))^{3/2} +  C F(t) e_r(t), \qquad 
F(t) := \|  \tilde{u} (t,\cdot)\|_{H^{r+1}} + \|  \tilde{b} (t,\cdot) \|_{H^{r+1}} .
$$
It follows that the estimate
$$
e_r(t) \leq 
\frac{e^{C\int_0^t F(s)\,\de s}}{\left( \frac{1}{\sqrt{e_r(0)}}  - \frac{C}{2} \int_0^t e^{\frac{C}{2} \int_0^s F(\tau) \, \de \tau  } \, \de s \right)^{2}},
$$
holds provided that the denominator is positive. Since the solution exists as long as the $H^r$ norm is bounded, the previous estimate completes the proof.
\end{proof}

\begin{rem}\label{rem:stabilità viscoso}
The stability result in Theorem \ref{thm:stab} also extends to the viscous \eqref{eq:mhd} system. Indeed, if viscosity is included in the equations, one simply needs to add the Laplacian term $\Delta v$ in the first equation of system \eqref{eq:mhdPerturbative}. This additional term can be estimated in the same way as the Laplacian term $\Delta m$, and thus the proof carries over without further modifications.
\end{rem}

\section{Accelerated reconnection scenario}\label{Section:Proof}
In this section, we present the proof of our first main result. We recall once again that our goal is to build examples of solutions of \eqref{eq:mhd} that show accelerated reconnection of magnetic lines. Moreover, we aim for this scenario to be stable with respect to small perturbations of the initial data and target times. To achieve this, our starting point will be the example from Section \ref{Sec:Bifurcation}. First of all, we introduce the following definition.

\begin{defn}\label{def:diss enhBis}
Let $\eta_0\in(0,1)$. An adapted function is a continuous increasing function $\Lambda: (0,\eta_0)\to(0,+\infty)$
such that 
\begin{equation}\label{lambdaBig}
\lim_{\eta\to 0^+} \Lambda(\eta) < \infty.
\end{equation}
\end{defn}

The examples of adapted functions that are relevant for us are $\Lambda(\eta) = C$, $\Lambda(\eta) = C \eta$ or $\Lambda(\eta) = \lambda(\eta)$, where $\lambda(\eta)$ is an enhanced dissipation rate as in Definition \ref{def:diss enh}.

Hereafter, for $F: [0,\delta] \times \T^d \to \R^d$, $\delta >0$, we will sometimes abbreviate 
\begin{equation}
\| F \|_{L^\infty_\delta X}  := \|F\|_{L^\infty((0,\delta); X )}  .
\end{equation}
where $X$ is a Sobolev or a Besov space. We can now prove the following.

\begin{mainthm3}\label{lem:enh Hr}
Let $U\in L^\infty((0,T); W^{S,\infty}(\T^d))$, for some integer $S\geq 1$ and $d\geq 2$, be a divergence-free vector field, and let $\rho$ be the solution of \eqref{eq:ad} corresponding to an initial datum $\rho^\mathrm{in}\in~H^r(\T^d)$, with $r$ being a positive integer, such that
\begin{equation}\label{passo 0}
\|\rho(t,\cdot)\|_{L^2}\lesssim e^{-\Lambda(\eta)t}\|\rho^\mathrm{in}\|_{L^2},
\end{equation}
where $\Lambda(\eta)$ is an adapted function as in Definition \ref{def:diss enh}. 
Then, there exists $\eta_0\in(0,1)$ such that the following holds: for all 
$\eta \in (0, \eta_0)$, for any $0\leq s\leq \min\{r,S\}$ and for all $\varepsilon >0$, there exists a constant $C_{\e,s}>0$ such that
\begin{equation}\label{eq:enhanced hr}
\|\rho(t,\cdot)\|_{H^s}\leq C_{\e,s} \frac{ 1+  \|U\|_{L^\infty_t W^{S,\infty}}^s}{\eta^{s\left(\frac12 + \varepsilon\right)}} e^{-\Lambda(\eta)t}\|\rho^\mathrm{in}\|_{H^s}, \qquad 0\leq t \leq T.
\end{equation}
The $\varepsilon$-loss can be removed if $U \in L^\infty((0,T);B^{S}_{\infty,2}(\T^d))$. In this case one has that there exists $C_s>0$ such that
\begin{equation}\label{eq:enhanced hrBesov}
\|\rho(t,\cdot)\|_{H^s} \leq C_{s} \frac{ 1+  \|U\|_{ L^\infty_t B^{S}_{\infty,2} }^s}{\eta^{s/2 }} e^{- \Lambda(\eta)t}\|\rho^\mathrm{in}\|_{H^s}, \qquad 0\leq t \leq T.
\end{equation}
\end{mainthm3}

We stress the fact that Theorem B' works for the case of non-autonomous vector fields, and thus in the context of exponential mixers. On the other hand, in the autonomous case the condition \eqref{passo 0} holds if $U$ is dissipation enhancing with rate $\lambda(\eta)$, taking $\Lambda(\eta)=\lambda(\eta)$, and the initial datum has zero streamline-average. Thus, Theorem B follows from Theorem B'.

\begin{rem}
Note that if $\Lambda(\eta) = \eta^{\alpha}$ with $\alpha \geq 1/2$, the inequality
\eqref{eq:enhanced hrBesov} implies (and in fact improves on for $\alpha > 1/2$) the following estimate
\begin{equation}\label{eq:enhanced hrCZ}
\|\rho(t,\cdot)\|_{H^s}\lesssim \frac{ 1+\|U\|_{L^\infty_{t} B^{S}_{\infty,2} }^s }{\Lambda(\eta)^{s}} e^{- \Lambda(\eta)t}\|\rho^\mathrm{in}\|_{H^s}, \qquad 0\leq t \leq T.
\end{equation}
\end{rem}

\begin{proof}
First of all, we prove the estimate \eqref{eq:enhanced hr} for integers values of $s$ by induction. The induction base $s=0$ follows trivially from the assumption \eqref{passo 0}. Thus, we consider $s\geq 1$ and we assume that \eqref{eq:enhanced hr} holds with $s-1$. We denote with $P_N$ the Littlewood-Paley projector on the dyadic scale $[N, 2N]$ with $N \in 2^{\mathbb{N}}$, and we define
\begin{equation}\label{defBetafdjsk}
\beta(\eta) := \min \left( \eta^{\frac12 + \varepsilon},  2 \left( \frac{\eta}{\Lambda(\eta)} \right)^{\frac12} \right).
\end{equation}

Recalling \eqref{lambdaBig}, we see that for all $\eta$ sufficiently small we have in fact $\beta(\eta) = \eta^{\frac12 + \varepsilon}$.
Thus we can rewrite \eqref{eq:enhanced hr} as
\begin{equation}\label{eq:enhanced hrBis}
\|\rho(t,\cdot)\|_{H^s}\lesssim_{\varepsilon, s} \frac{ 1+C_{\varepsilon, s} \|U\|_{L^\infty_t W^{S,\infty}}^s} { \beta(\eta)^s }  e^{- \Lambda(\eta)t}\|\rho^\mathrm{in}\|_{H^s}.
\end{equation}
By using \eqref{definizione besov}, we can estimate the $H^s$ norm of $\rho$ as 
\begin{align}
\| \rho(t,\cdot) \|_{H^s} &\lesssim \frac{1}{\beta(\eta)^{s}}   \| \rho(t,\cdot) \|_{L^2}
+ \left( \sum_{N > \frac{1}{\beta(\eta)}}N^{2s} \| P_N \rho(t,\cdot) \|^2_{L^2} \right)^{1/2}\nonumber\\
&\leq \frac{1}{\beta(\eta)^{s}}e^{-\Lambda(\eta)t}\|\rho^\mathrm{in}\|_{L^2}+ \left( \sum_{N > \frac{1}{\beta(\eta)}}N^{2s} \| P_N \rho(t,\cdot) \|^2_{L^2} \right)^{1/2}
\label{splitting}
\end{align}
where the bound for the low frequency contribution follows by \eqref{passo 0}.\\

Let $c\ll 1$ be a suitable small constant (to be fixed later). We split the velocity field $U$ as
\begin{equation}
U:=U_1+U_2,
\end{equation}
where 
\begin{equation}\label{proiettori smooth}
U_1(t,\cdot):= P_{\leq cN} U(t,\cdot),\qquad 
U_2(t,\cdot):= P_{> cN} U (t, \cdot).
\end{equation}
Note that $U_1,U_2$ depends on $N$, although we do not keep track of this dependence in the notation. 
To estimate the high frequencies we first apply $P_N$ to \eqref{eq:ad} and we get that
\begin{equation}\label{TransportForRhoProjected}
 \partial_t P_N \rho  +  (U_1 \cdot \nabla ) P_N \rho  - \eta \Delta P_N \rho = -P_N\left((U_2 \cdot \nabla)  \rho \right)+
    [ U_1 \cdot \nabla , P_N ] \rho .
\end{equation}
Taking the scalar product of \eqref{TransportForRhoProjected} 
against $N^{2s} P_N \rho$, integrating by parts, and exploiting the divergence-free condition on $U_1$ we arrive to
\begin{align*}
 \frac12\frac{\de}{\de t}  N^{2s}  \| P_N \rho (t,\cdot) \|_{L^2}^2 
+ \eta N^{2s}  \| \nabla P_{N} \rho(t,\cdot) \|^2_{L^2}&= 
 N^{2s}  \int_{\T^d} [ U_1 \cdot \nabla , P_N ] \rho(t,x)   P_{N}\rho(t,x) \de x\\
&-N^{2s}\int_{\T^d}P_N\left((U_2 \cdot \nabla ) \rho(t,x) \right)P_N \rho(t,x) \de x.
\end{align*}
We start by considering the term involving $U_1$: we define the operator $\tilde P_N$ as
$$
\tilde P_N f:=\sum_{\frac{N}{C}\leq L\leq CN}P_L f,
$$
where $C>0$ is some positive constant that has to be chosen big enough in order to have that
$$
[(U_1\cdot\nabla), P_N]\rho=[(U_1\cdot\nabla), P_N]\tilde P_N\rho.
$$
This can be done using $P_N\tilde P_N=P_N$ and the following Fourier support consideration:
\begin{align*}
\supp \mathcal{F}\left((U_1\cdot\nabla)(I-\tilde{P}_N)\rho\right)&\subset \supp \mathcal{F}(U_1)+\supp \mathcal{F}((I-\tilde{P}_N)\rho)\\
&\subset \left(B_{cN}+B_{\frac{N}{C}}\right)\cup\left((\R^d\setminus B_{CN})+B_{cN}\right),
\end{align*}
so it is enough to take $c\ll 1$ and $C\gg 1$ in order to make this set disjoint from the support of $P_N$.
Thus, we use the commutator estimate in Lemma \ref{lem:commutatore} and Young inequality to obtain
\begin{align*}
N^{2s}  \int_{\T^d} [ U_1 \cdot \nabla , P_N ] \rho(t,x)   P_{N}\rho(t,x) \de x&=N^{2s}  \int_{\T^d} [ U_1 \cdot \nabla , P_N ] \tilde P_N\rho(t,x)   P_{N}\rho(t,x) \de x\\
&\leq N^{2s} \|\nabla U_1(t,\cdot)\|_{L^\infty}\|\tilde P_N\rho(t,\cdot)\|_{L^2}\|P_N\rho(t,\cdot)\|_{L^2}\\
&\leq \frac{C}{\eta}N^{2(s-1)} \|\nabla U_1(t,\cdot)\|_{L^\infty}^2\|\tilde P_N\rho(t,\cdot)\|_{L^2}^2+\frac{\eta}{4} N^{2(s+1)}\|P_N\rho(t,\cdot)\|_{L^2}^2.
\end{align*} 

We now consider the term involving $U_2$: by H\"older and Young inequality we get that
\begin{align*}
N^{2s}\int_{\T^d}P_N\left(U_2 \cdot \nabla \rho  (t,x)\right)P_N \rho (t,x) \,\de x
&=-N^{2s}\int_{\T^d}P_N\left(U_2 \rho (t,x) \right)\cdot \nabla P_N \rho (t,x) \,\de x\\
&\leq N^{2s}\|P_N\left(U_2\rho (t,x) \right)\|_{L^2}\|\nabla P_N\rho (t,x)\|_{L^2}\\
&\leq \frac{\eta}{4} N^{2(s+1)}\|P_N\rho(t,\cdot)\|_{L^2}^2+\frac{C}{\eta}N^{2s}\|P_N\left(U_2\rho (t,\cdot)\right)\|_{L^2}^2\\
&\leq \frac{\eta}{4} N^{2(s+1)}\|P_N\rho(t,\cdot)\|_{L^2}^2+\frac{C}{\eta}N^{2s}\|U_2(t,\cdot)\|_{L^\infty}^2\|\rho(t,\cdot)\|_{L^2}^2.
\end{align*}
Then, we use the diffusion to absorb on the left hand side the terms with powers $N^{2(s+1)}$, obtaining
\begin{align}\label{Plugintosajsfdaskdjnfg}
\frac{\de}{\de t}  N^{2s}  \| P_N \rho (t,\cdot) \|_{L^2}^2 +\eta N^{2(s+1)}\| P_N \rho (t,\cdot) \|_{L^2}^2&\leq \frac{C}{\eta}N^{2(s-1)} \|\nabla U_1\|_{L^\infty}^2\|\tilde P_N\rho(t,\cdot)\|_{L^2}^2\\ \nonumber
&+\frac{C}{\eta}N^{2s}\|U_2(t,\cdot)\|_{L^\infty}^2\|\rho(t,\cdot)\|_{L^2}^2.
\end{align}
Using the following estimate 
\begin{align*}
N^{2(s-1)}\|\tilde P_N\rho(t,\cdot)\|_{L^2}^2&=N^{2(s-1)}\sum_{\frac{N}{C}\leq M\leq CN}\|P_M\rho(t,\cdot)\|_{L^2}^2\\
&\leq C^{2(s-1)}\sum_{\frac{N}{C}\leq M\leq CN}M^{2(s-1)}\|P_M\rho(t,\cdot)\|_{L^2}^2 \lesssim\|\rho(t,\cdot)\|_{H^{s-1}}^2,
\end{align*}
from the induction hypothesis we deduce 
\begin{equation}\label{est:bern-u2bdhsjhb}
N^{s-1} \|\tilde P_N\rho(t,\cdot)\|_{L^2} \lesssim 
\frac{1 + \|U\|_{L^\infty_t W^{S,\infty}}^{s-1}}{\beta(\eta)^{s-1}}e^{- 2 \Lambda(\eta) t}\|\rho^\mathrm{in}\|_{H^{s-1}}.
\end{equation}
On the other hand, by the definition of $U_2$ and the Bernstein inequality \eqref{eq:bern1} we have
\begin{equation}\label{est:bern-u2}
\| U_2(t,\cdot)\|_{L^\infty}\leq N^{-S}\|U_2(t,\cdot)\|_{W^{S,\infty}}, \qquad \mbox{for a.e. }t\in[0,T]. 
\end{equation}
Plugging the estimates \eqref{est:bern-u2bdhsjhb}, \eqref{est:bern-u2} into \eqref{Plugintosajsfdaskdjnfg}
and using the enhanced dissipation bound \eqref{passo 0} we arrive to 
\begin{align*}
&\frac{\de}{\de t}  N^{2s}  \| P_N \rho (t,\cdot) \|_{L^2}^2
\\
&
\leq -\eta N^2 N^{2s}\| P_N \rho (t,\cdot) \|_{L^2}^2 +
\frac{C}{\eta}\|\nabla U_1(t,\cdot) \|_{L^\infty}^2 (1+\|U\|_{L^\infty_t W^{S,\infty}}^{2(s-1)})\frac{C}{\beta(\eta)^{2(s-1)}}e^{-2\Lambda(\eta)t}\|\rho^\mathrm{in}\|_{H^{s-1}}^2,
\\
& +\frac{C}{\eta}N^{2(s-S)}\|U_2 \|_{L^\infty_t W^{S,\infty}}^2e^{- 2 \Lambda(\eta) t}\|\rho^\mathrm{in}\|_{L^2}^2.
\end{align*} 
By defining $h_s^N(t):=N^{2s}  \| P_N \rho (t,\cdot) \|_{L^2}^2$, we apply Gr\"onwall inequality and the crude estimate
$$
\|\nabla U_1\|_{L^\infty_t L^{\infty}}^2(1+\|U\|_{L^\infty_t W^{S,\infty}}^{2(s-1)}) \lesssim
1+\|U\|_{L^\infty_t W^{S,\infty}}^{2s},
$$ 
to get
\begin{align*}
 h_s^N(t)& 
\leq 
\frac{C}{\eta \beta(\eta)^{2(s-1)}} (1+\|U\|_{L^\infty_t W^{S,\infty}}^{2s}) \|\rho^\mathrm{in}\|_{H^{s-1}}^2e^{-  \eta N^{2} t}\int_0^t e^{ \eta N^{2} \tau}e^{- 2 C_1 \Lambda(\eta) \tau}\de \tau
\\
&
+
e^{-  \eta N^{2} t} h_s^N(0) +\frac{C}{\eta}N^{2(s-S)}\|U_2\|_{L^\infty_t W^{S,\infty}}^2\|\rho^\mathrm{in}\|_{L^2}^2 e^{- \eta N^{2} t} \int_0^t e^{ \eta N^{2} \tau}e^{- 2 \Lambda(\eta) \tau}\de \tau.
\end{align*}
Recalling \eqref{defBetafdjsk} and the fact that we restrict to $N \geq \frac{1}{\beta(\eta)}$, we have
\begin{equation}\label{fine}
\eta N^2\geq  \frac{\eta}{\beta(\eta)^2}  \geq 4\Lambda(\eta),
\end{equation}
and as a consequence
$$
\int_0^t e^{ \eta N^{2} \tau}e^{- 2 \Lambda(\eta) \tau}\de \tau=\frac{e^{ \eta N^{2}t-2 \Lambda(\eta)t}-1}{\eta N^{2}-2 \Lambda(\eta)}\leq \frac{2}{\eta N^2}e^{ \eta N^{2}t-2 \Lambda(\eta)t}.
$$
This latter implies that (here $\varepsilon' > 0$)
\begin{align}
\nonumber & h_s^N(t) \leq e^{-\eta N^{2} t} h_s^N(0)
\\ \nonumber
&+\left(\frac{C(1 + \|U\|_{L^\infty_t W^{S,\infty}}^{2s})}{\eta^{2}N^2\beta(\eta)^{2(s-1)} } + 
\frac{C}{\eta^{2}}N^{2(s-S-1)}\|U_2\|_{L^\infty_t W^{S,\infty}}^2\right)\|\rho^\mathrm{in}\|_{H^{s-1}}^2
e^{- 2 \Lambda(\eta) t} \\ \nonumber
& \leq e^{-\eta N^{2} t} h_s^N(0)
\\
&  
+\left(\frac{C}{\eta^{2} N^{\varepsilon'} \beta(\eta)^{2s - 4 + \varepsilon'}} + \frac{C}{\eta^{2} \beta(\eta)^{-2 + 2 \varepsilon'}}N^{2(s-S-\varepsilon')}\right)(1 + \|U\|_{L^\infty_t W^{S,\infty}}^{2s})\|\rho^\mathrm{in}\|_{H^{s-1}}^2e^{- 2 \Lambda(\eta) t},\label{norme alte}
\end{align}
where in the last line we used $N \geq \frac{1}{\beta(\eta)}$, $s \geq 1$ and we exploited the fact that $U_1,U_2$ are defined in \eqref{proiettori smooth} via smooth projectors.

Recalling that $\beta=\eta^{\frac12 + \varepsilon}$ and that we restrict to $s \geq 1$, we see that for all 
$\varepsilon' = \varepsilon'(\varepsilon) >0$ and 
$\eta >0$ sufficiently small we have
$$
\frac{1}{\eta^{2}  \beta(\eta)^{2s - 4 + \varepsilon'}} \leq \frac{1}{ \beta(\eta)^{2s}}, \qquad 
\frac{1}{\eta^{2} \beta(\eta)^{-2 + 2 \varepsilon'}} \leq \frac{1}{ \beta(\eta)^{2s}},
$$
so that we have arrived to
\begin{align*}
 h_s^N(t) \leq e^{-\eta N^{2} t} h_s^N(0)+ \frac{C_{s}}{N^{\varepsilon' }\beta(\eta)^{2s}} 
 (1 + \|U\|_{L^\infty_t W^{S,\infty}}^{2s})\|\rho^\mathrm{in}\|_{H^{s-1}}^2e^{- 2 \Lambda(\eta) t}.
\end{align*}
Lastly, using again \eqref{fine}, summing over $N>\frac{1}{\beta(\eta)}$, and recalling \eqref{defBetafdjsk} we get that 
\begin{equation}
\sum_{N > \frac{1}{\beta(\eta)}}N^{2s} \| P_N \rho(t,\cdot) \|^2_{L^2} 
\leq
\frac{C_{s, \varepsilon}}{\beta(\eta)^{2s}}
(1 + \|U\|_{L^\infty_t W^{S,\infty}}^{2s})\|\rho^\mathrm{in}\|_{H^{s-1}}^2 e^{- 2 \Lambda(\eta) t},
\end{equation}
as long as $s \leq S$. Plugging this into \eqref{splitting} the inequality \eqref{eq:enhanced hr} follows. Finally, the result for non-integer values of $s$ follows by applying the Riesz-Thorin interpolation theorem to the linear bounded operator
$$
T:\rho^\mathrm{in}\in H^s(\T^d)\mapsto \rho(t,\cdot)\in H^s(\T^d),
$$
where $\rho(t,\cdot)$ is the unique solution of \eqref{eq:ad} with initial datum $\rho^\mathrm{in}$.\\
\\
In order to prove the estimate \eqref{eq:enhanced hrBesov}, in which the $\varepsilon$-loss is removed, we only need to modify the part of the argument which involves $U_1$. Since the statement is trivial for $s=0$, we can assume that $s$ is an integer in $[1, \min(r,S)]$. 
Using Lemma 2.100 in \cite{BCD} it follows that
\begin{equation}
\|[U_1\cdot \nabla, P_N]\tilde P_N\rho(t,\cdot)\|_{L^2}\leq N^{-s} c_N \underbrace{\left(\|\nabla U_1(t,\cdot)\|_{L^\infty}\|\tilde P_N\rho(t,\cdot)\|_{H^s}+\|\nabla\tilde P_N\rho(t,\cdot)\|_{L^2}\|\nabla U_1(t,\cdot)\|_{B^{s-1}_{\infty, 2}}\right)}_{:=K_{N}[U_1,\rho](t)}, 
\end{equation}
where 
\begin{equation}\label{cNL2}
\sum_{N \in 2^{\mathbb{N}}} c_N^2 \lesssim 1.
\end{equation}
So, we have that
\begin{align*}
N^{2s}  \int_{\T^d} [ U_1 \cdot \nabla , P_N ] \rho(t,x)   P_{N}\rho(t,x) \de x&=N^{2s}  \int_{\T^d} [ U_1 \cdot \nabla , P_N ] \tilde P_N\rho(t,x)   P_{N}\rho(t,x) \de x\\
&\leq N^{2s} N^{-s} c_N K_{N}[U_1,\rho](t)\|P_N\rho(t,\cdot)\|_{L^2}\\
&\leq \frac{C}{\eta}N^{-2} c_N^2 K_{N}[U_1,\rho](t)^2+\frac{\eta}{4} N^{2(s+1)}\|P_N\rho(t,\cdot)\|_{L^2}^2.
\end{align*}
Since $s \geq 1$, we use that
$$
\|\nabla\tilde P_N\rho(t,\cdot)\|_{L^2}\leq \|\tilde P_N\rho(t,\cdot)\|_{H^s},
$$
$$
\|\nabla U_1(t,\cdot)\|_{B^{s-1}_{\infty, 2}} \lesssim \| U_1(t,\cdot) \|_{B^{s}_{\infty, 2}} \lesssim \| U(t,\cdot)\|_{B^{S}_{\infty, 2}},
$$
to deduce
\begin{equation}\label{combfin1}
|K_{N}[U_1,\rho](t)| \lesssim  \|\tilde P_N\rho(t,\cdot)\|_{H^s} \| U(t,\cdot)\|_{B^{S}_{\infty, 2}} .
\end{equation}
Then, as before, we note
\begin{align*}
N^{-2} \|\tilde P_N\rho(t,\cdot)\|_{H^s}^2&=N^{-2}\sum_{\frac{N}{C}\leq M\leq CN}M^{2s}\|P_M\rho(t,\cdot)\|_{L^2}^2\\
&\leq C^{2}\sum_{\frac{N}{C}\leq M\leq CN}M^{2(s-1)}\|P_M\rho(t,\cdot)\|_{L^2}^2 \lesssim\|\rho(t,\cdot)\|_{H^{s-1}}^2,
\end{align*}
and we use the (new) induction hypothesis to obtain 
\begin{equation}\label{combfin2}
N^{-1}\|\tilde P_N\rho(t,\cdot)\|_{H^s}\lesssim 
\frac{1 + \|U\|_{L^\infty_t B^{S}_{\infty, 2}}^{s-1}}{\beta(\eta)^{s-1}}e^{- 2 \Lambda(\eta) t}\|\rho^\mathrm{in}\|_{H^{s-1}},
\end{equation}
Combining \eqref{combfin1}-\eqref{combfin2} we get that  
$$
N^{-1} |K_{N}[U_1,\rho](t)| \lesssim  
\frac{1 + \|U\|_{L^\infty_t B^{S}_{\infty, 2}}^{s}}{\beta(\eta)^{s-1}}e^{- 2 \Lambda(\eta) t}\|\rho^\mathrm{in}\|_{H^{s-1}},
$$ 
and then we proceed exactly as in the proof of \eqref{eq:enhanced hr}, but we use \eqref{cNL2} to sum over $N$, instead of the factor $N^{-\varepsilon'}$.
\end{proof}

\begin{rem}
In the case of a shear flow, one can obtain a similar estimate with a different proof. Indeed, by exploiting the fact that $\partial_x \rho$ solves the equation \eqref{eq:ad}, it is straightforward to prove that $\partial_x^s \rho$ are dissipated with the same rate of $\rho$ (no $\eta$-corrections in the multiplicative constant). On the other hand, differentiating the equation with respect to y produces an additional term. This can be handled using the Duhamel formula and the enhanced dissipation property of the semigroup generated by $\partial_t+U\cdot\nabla-\eta\Delta$. Then, one has to iterate the procedure for higher derivatives, obtaining a multiplicative constant that grows like $(1+t)^s$, which is better than \eqref{eq:enhanced hrBesov-intro} for $t\lesssim \eta^{-1/2}$. We thank Michele Dolce and Michele Coti Zelati for pointing this out.
\end{rem}

We can finally prove our main result, we recall the precise statement below.
\begin{mainthm}
Let $r > 5/2$. There exists $\eta_0 >0$ sufficiently small such that there exists a family of divergence-free vector fields $(u^\mathrm{in}_\eta,b^\mathrm{in}_\eta)$ in $H^r(\T^3)$ with $0<\eta<\eta_0$, such that the unique local solution $(u_{\eta},b_{\eta},p_{\eta})$ of \eqref{eq:mhd} arising from $(u^\mathrm{in}_\eta,b^\mathrm{in}_\eta)$ shows accelerated magnetic reconnection with rate
$$
\lambda_\mathrm{rec}(\eta) := c_2 \frac{\eta^\frac12}{|\ln \eta|},
$$
for some positive constant $c_2>0$. More precisely, the magnetic field $b_{\eta}(t,\cdot)$ has an equilibrium point for times 
$t \in [0, c_1]$ and has no equilibrium points for times $t \in \left[c_2 \frac{|\ln \eta|}{\eta^\frac12}, 2 c_2 \frac{|\ln \eta|}{\eta^\frac12}\right]$, where the constants $c_1, c_2$ only depends on $r$. 
The statement is structurally stable, in the sense that it still holds 
for all solutions with (divergence-free) initial data $(\tilde u^\mathrm{in}_\eta, \tilde b^\mathrm{in}_\eta)$
such that 
\begin{equation}\label{eqPertIntro2}
\| \tilde u^\mathrm{in}_\eta - u^\mathrm{in}_\eta\|_{H^r} 
 + \|\tilde b^\mathrm{in}_\eta -  b^\mathrm{in}_\eta \|_{H^r} \leq c(r, \eta)
\end{equation}
with $c(r, \eta)$ sufficiently small.
\end{mainthm}
\begin{proof}
We divide the proof in several steps.\\
\\
\underline{\em Step 1} Reference solution.\\
\\
We consider an explicit solution of \eqref{eq:mhd} of the following form: define $(\tilde u,\tilde b)$ as
\begin{equation}\label{DefUBAgain}
\tilde{u} := ( U_1, U_2, 0), 
\qquad \tilde{b} := (0, 0, \tilde{b}_3),
\end{equation}
where 
\begin{itemize}
\item $U:=(U_1, U_2) : (x_1, x_2) \in \T^2 \to \R^2$ is a stationary solution of the 2D Euler equations;
\item $U$ is dissipation enhancing with rate $\lambda(\eta)$;
\item $\tilde{b}_3 : (t,x_1, x_2) \in [0, \infty)\times \T^2   \to \R 
$ is a solution of the advection-diffusion equation 
\begin{equation}\label{AdvDiffForB3Again}
\partial_t \tilde{b}_3   + (U \cdot \nabla ) \tilde{b}_3 = \eta \Delta \tilde{b}_3.
\end{equation}
\end{itemize}
We couple the equation \eqref{AdvDiffForB3Again} with an initial condition 
$$
\tilde{b}_3\big|_{t=0}=\tilde{b}^\mathrm{in}_3,
$$ 
where the function $\tilde{b}^\mathrm{in}_3:\T^2\to\R$ is chosen so that:
\begin{itemize}
\item there exists $(x_1^*,x_2^*)\in \T^2$ such that $\tilde{b}^\mathrm{in}_3(x_1^*,x_2^*)=0$ 
and $\partial_{x_1}\tilde{b}^\mathrm{in}_3(x_1^*,x_2^*)\simeq M$;
\item it has strictly positive average, i.e. 
\begin{equation}\label{fkaldkfalkfmkasl}
\langle \tilde{b}^\mathrm{in}_3 \rangle = M,
\end{equation}
for some given $M>0$
and 
$$
\| \tilde{b}^\mathrm{in}_3\|_{H^r} \simeq M;
$$
\item it can be written as $\tilde{b}^\mathrm{in}_3:=M+\rho^\mathrm{in}$, with $\rho^\mathrm{in}$ belonging to the space $H_\mathrm{shear}$, as in \eqref{def:hshear}.
\end{itemize}
The couple \eqref{DefUBAgain} will be our reference solution. 
A solution of the type \eqref{DefUBAgain} which satisfy all the above conditions is the following: we take as initial datum of \eqref{AdvDiffForB3Again}, the function
\begin{equation}\label{esempio bin}
\tilde b^\mathrm{in}_3=M-2M\cos \left(x_1 - \frac{\pi}{3}-x_1^*\right),
\end{equation}
with the velocity field chosen as a shear flow of the form
$$
U(x_1,x_2)=(f(x_2),0),
$$ 
for some $f\in C^3(\T)$ satisfying the assumptions of Theorem \ref{teo:bcz} with $n_0=2$. For example, one can consider the Kolmogorov flow, i.e. $f(x_2)=\sin x_2$. It follows that the rate of enhanced dissipation in the functional space $H_\mathrm{shear}$, is given by
\begin{equation}
\lambda(\eta)=C\,\eta^\frac12,
\end{equation}
for some constant $C>0$.\\
\\
\underline{\em Step 2} Perturbative argument and estimates on Sobolev norms.\\
\\
We consider a divergence-free initial datum $(b^\mathrm{in}, u^\mathrm{in})$ of the following form: 
\begin{align*}
b^\mathrm{in}(x_1, x_2, x_3) &= \tilde{b}^\mathrm{in}(x_1, x_2) +  m^\mathrm{in} (x_1, x_2, x_3),\\
u^\mathrm{in} &= \tilde{u}^\mathrm{in},
\end{align*}
where $\tilde u$ and $\tilde b$ are defined as in the previous step, while $m^\mathrm{in} (x_1, x_2, x_3)$ is a zero average smooth divergence-free vector field with a regular equilibrium point at $(x_1^*,x_2^*,x_3^*)\in \T^3 $ and such that 
$$
\| m^\mathrm{in} \|_{H^r} \simeq \varepsilon \ll 1,
$$ 
for some $r > 5/2$. The parameter $\e$ will be chosen later. We will also choose $m^\mathrm{in}$ in such a way that $(x_1^*, x_2^*, x_3^*)$ is an hyperbolic equilibrium of $b^{in}$ with 
\begin{equation}\label{fjdskjhgn}
\nabla b^\mathrm{in} \, (x_1^*,x_2^*,x_3^*) = \left(
\begin{array}{ccc}
0 & \varepsilon & 0 \\
0 & 0 & \varepsilon \\
c M  & 0 & 0
\end{array}
\right)
\end{equation}
for some absolute constant $c \neq 0$.
In particular 
$$
\det ( \nabla b^\mathrm{in}) \big|_{(x_1^*,x_2^*,x_3^*)} = c \varepsilon^2M.
$$
thus the equilibrium $(x_1^*,x_2^*,x_3^*)$ is regular.
To simplify the argument we consider an initial perturbation of the form
\begin{equation}\label{semplificazione m}
m_1^\mathrm{in}=m_1^\mathrm{in}(x_2),\qquad m_2^\mathrm{in}=m^\mathrm{in}_2(x_3), \qquad m_3^\mathrm{in}=0
\end{equation}
satisfying
\begin{equation}\label{semplificazione m 2}
\frac{\de}{\de x_2}m_1^\mathrm{in}(x_2^*) = \frac{\de}{\de x_3}m^\mathrm{in}_2(x_3^*) = \e.
\end{equation}
For example, if we set $b_3^\mathrm{in}$ as in \eqref{esempio bin}, we can consider
\begin{equation}\label{Minitial}
m^\mathrm{in}(x_1,x_2,x_3)=\e(\sin (x_2- x_2^*), \sin (x_3 - x_3^*), 0).
\end{equation}
Note that under these choices \eqref{fjdskjhgn} holds with $c = - \sqrt{3}$.

Under the above assumptions we have a local solution $(u,b)$ of the \eqref{eq:mhd} equations arising from $(u^\mathrm{in}, b^\mathrm{in})$ with life-span quantified by Theorem \ref{thm:stab}. In particular, for any $T>0$ we can choose $\e$ small enough so that the solution exists up to time $T$. This is crucial because it ensures that the (estimated) reconnection time remains within the time range of local existence. To conclude this step we provide estimates on the Sobolev norms of our solution that will be frequently used in the sequel. \\

We start by considering the reference solution $(\tilde u,\tilde b)$. 
First of all, we have that 
\begin{equation}\label{fdhsjkjdhgfU}
\|\tilde u\|_{H^r}\simeq \|\tilde u\|_{B^r_{\infty,2}}\simeq 1.
\end{equation} 
For the magnetic field, we can rewrite $\tilde{b}^\mathrm{in}_3:=M+\rho^\mathrm{in}$ with $\rho^\mathrm{in}\in H_\mathrm{shear}$ such that $\|\rho^\mathrm{in}\|_{H^r} \simeq M$.
Next, we rewrite $\tilde b_3$ as 
\begin{equation}\label{fjdksldkfjgksldkgjksl}
\tilde{b}_3(t,x) = \langle \tilde{b}_3 \rangle + \rho(t,x),
\end{equation}
where $\rho$, the zero-mean part of the tracer, satisfies the advection-diffusion equation
\begin{equation}\label{TransportForRho}
\begin{cases}
\partial_t \rho  + (U \cdot \nabla )  \rho = \eta \Delta  \rho,\\
\rho_{|t=0}=\rho^\mathrm{in}.
\end{cases}
\end{equation}
This follows because $\tilde{b}_3$ also satisfies this equation. Since $ \rho^\mathrm{in} \in H_\mathrm{shear}$ we have the following enhanced diffusion inequality 
\begin{equation}\label{EnhDiff} 
\| \rho(t,\cdot)  \|_{L^2} \leq C e^{-\lambda(\eta) t} \| \rho^\mathrm{in} \|_{L^2},
\end{equation}
for all $t\geq 0$. Thus, since $\|\rho^\mathrm{in}\|_{H^r}\simeq M$, we have that
\begin{align}
\|\rho(t,\cdot)\|_{H^r}&\leq Me^t,\label{stima rho 1}\\
\|\rho(t,\cdot)\|_{H^r}&\leq \frac{C M}{\eta^{\frac{r}{2}}} e^{- \lambda(\eta) t}\label{stima rho2},
\end{align}
where \eqref{stima rho 1} follows from the fact that $\rho$ solves the linear equation \eqref{TransportForRho} and $U$ is time-independent, while \eqref{stima rho2} is a consequence of Theorem B. We point out that the estimate \eqref{stima rho 1} is better than \eqref{stima rho2} for
$$
0\leq t\leq \bar{T}_\eta:=\frac{1}{1+\lambda(\eta)}\ln\left(\frac{C}{\eta^{\frac{r}{2}}}\right).
$$
By using this consideration, we find that
\begin{equation}
f(T):=\frac{1}{\e}-\frac{C}{2} \int_0^T e^{\frac{C}{2} \int_0^s \|  \tilde{u}(\tau, \cdot) \|_{H^{r+1}} + \|  \tilde{b} (\tau, \cdot) \|_{H^{r+1}}   \, \de \tau  } \, \de s\geq \frac{1}{2\e},
\end{equation}
as long as
\begin{equation}\label{scelta epsilon}
\e\lesssim Me^{-CM\eta^{-\frac{r}{2}}}e^{-MT}.
\end{equation}
Indeed, we notice that
\begin{align}
\int_0^t \|  \tilde{u}(\tau, \cdot) \|_{H^{r+1}} + \|  \tilde{b} (\tau, \cdot) \|_{H^{r+1}}   \, \de \tau\lesssim Mt+M(e^t-1),&\qquad \mbox{ for }0\leq t\leq \bar{T}_\eta,\\
\int_0^t \|  \tilde{u}(\tau, \cdot) \|_{H^{r+1}} + \|  \tilde{b} (\tau, \cdot) \|_{H^{r+1}}   \, \de \tau\lesssim Mt+M(e^{\bar{T}_\eta}-1)+\frac{CM}{\eta^{\frac{r}{2} }},&\qquad \mbox{ for }\bar{T}_\eta\leq t< T.
\end{align}
So, we apply Theorem \ref{thm:stab} and we find that 
\begin{equation}\label{bound m tempi piccoli}
\|v(t, \cdot) \|_{H^r} + \|m(t,\cdot)\|_{H^r}\lesssim \e, \qquad \mbox{for }0\leq t\ll 1/M,
\end{equation}
and by \eqref{scelta epsilon} we also have that
\begin{equation}\label{stima m tutti i tempi}
\|v(t, \cdot) \|_{H^r} + \|m(t,\cdot)\|_{H^r}\lesssim \e\, e^{M t}e^{CM\eta^{-\frac{r}{2}}}\ll M, \qquad \mbox{for } \bar{T}_\eta\leq t< T.
\end{equation}
\\
\\
\underline{\em Step 3} Small times regime.\\
\\
The stability of hyperbolic equilibria implies that the solution $b(t,\cdot)$ has at least one hyperbolic equilibrium for $t$ sufficiently small. In this step we show that the solution retains one hyperbolic equilibrium for $0\leq t\leq c$ for some small constant $c$ which does not depend on $\eta$. First of all, we consider the reference solution $(\tilde u,\tilde b)$ and we recall that 
$$
\tilde b(t,x)=(0,0,M+\rho(t,x)).
$$
We rewrite the third component as follows
$$
\tilde b_3(t,x)=M+(\rho(t,x)-\rho^\mathrm{in}(x))+\rho^\mathrm{in}(x).
$$
By using the equation \eqref{TransportForRho}, together with the bound \eqref{stima rho 1}, for $0\leq t\ll 1$ (independent on $\eta$) it follows that 
\begin{align}
\|\rho(t,\cdot)-\rho^\mathrm{in}\|_{L^\infty}&\leq \int_0^t (\|(U\cdot\nabla)\rho(s,\cdot)\|_{L^\infty}+\eta\|\Delta\rho(s,\cdot)\|_{L^\infty})\de s\nonumber\\
&\leq M(1+\eta) (e^t-1)\leq M(1+\eta)t.\label{faccio uscire t}
\end{align}
as well as
\begin{align}
\|\partial_{x_1} \rho(t,\cdot)- \partial_{x_1} \rho^\mathrm{in}\|_{L^\infty}&
\leq \int_0^t (\|(U\cdot\nabla) \partial_{x_1} \rho(s,\cdot)\|_{L^\infty} + 
 \|(\partial_{x_1}U\cdot\nabla)\rho(s,\cdot)\|_{L^\infty}
+\eta\|\Delta\partial_{x_1}\rho(s,\cdot)\|_{L^\infty})\de s\nonumber\\
&\leq M(1+\eta) (e^t-1)\leq M(1+\eta)t.\label{faccio uscire t}
\end{align}
Thus, recalling our choice of $\rho^\mathrm{in}$ we have in fact 
\begin{equation*}
\tilde b_3(t,x_1,x_2) = M + \rho^\mathrm{in}(x_1) +  O\big(M(1 + \eta) t \big) = 
M - 2M\cos \left(x_1 - \frac{\pi}{3}-x_1^*\right) + O\big(M(1+ \eta) t\big)
\end{equation*}
as well as
\begin{align}\label{Monotonicity}
\partial_{x_1} \tilde b_3(t,x_1,x_2) 
& = \partial_{x_1} \rho(t,x_1, x_2) 
\\ & \nonumber
= \partial_{x_1} \rho^\mathrm{in}(x_1, x_2) + O \big(M(1 + \eta)t  \big) = 2 M \sin  \left(x_1 - \frac{\pi}{3}-x_1^*\right)   
+ O \big( M(1 + \eta) t \big). 
\end{align}
We now consider the solution $(u,b)$: we can write, for every $x\in\T^3$
$$
b(t,x)=(m_1(t,x),m_2(t,x),\tilde b_3(t,x_1,x_2)+m_3(t,x)).
$$
Using the equation for $m$, similarly to what we have done for $\rho$ in \eqref{faccio uscire t}, it is straightforward to show that
\begin{equation}\label{t per m}
\|m(t,\cdot)-m^\mathrm{in}(\cdot)\|_{H^r}\leq \e CM t.
\end{equation}
Since $b_3 = \tilde b_3 + m_3$ and $m_3^\mathrm{in}= 0$, using \eqref{faccio uscire t}-\eqref{Monotonicity}-\eqref{t per m}  
we deduce
\begin{align}\label{fndjskdjgnsk1}
b_3(t,x_1,x_2, x_3) & = M + \rho^\mathrm{in}(x_1) +  O\big(M(1 + \eta)t\big) +  m_3(t,\cdot)-m_3^\mathrm{in}(\cdot) 
\\ \nonumber
&
= 
M - 2M\cos \left(x_1 - \frac{\pi}{3}-x_1^*\right) + O\big( M(1+\varepsilon+ \eta) t\big)
\end{align}
as well as
\begin{align}\label{fndjskdjgnsk2}
 \partial_{x_1}  b_3(t,x_1,x_2,x_3) 
& 
= \partial_{x_1} \rho^\mathrm{in}(x_1, x_2) + O \big(M (1 + \eta)t  \big) 
+ \partial_{x_1} m_3(t, \cdot) - \partial_{x_1} m_3^\mathrm{in}(\cdot) 
\\ & \nonumber = 2 M \sin  \left(x_1 - \frac{\pi}{3}-x_1^*\right)  
+ O \big( M(1 + \varepsilon + \eta) t \big). 
\end{align}
Thus, by a continuity/monotonicity argument we see that there exists $c >0$ sufficiently small such that for all
 $(t, y, z) \in [0, c] \times \T^2$ there exists a unique function $x_1(t,y,z)$ with values in $[x_1^* - \frac{\pi}{6}, x_1^*  + \frac{\pi}{6}]$  
 which satisfies  
\begin{equation}\label{funzione implicita}
b_3(t,x_1(t,y,z),y,z)=0.
\end{equation}

By the uniqueness of $x_1(t,y,z)$ and the identity  
$$
b_3(0,x_1(0,y,z),y,z)= \tilde b^\mathrm{in}_3(0, x_1(0,y,z) ) = M-2M\cos \left(x_1(0,y,z) - \frac{\pi}{3}-x_1^*\right),$$
 we deduce that
\begin{equation}\label{fdhsjkdhbgfhjskhdbghjsdhbgfhjs}
x_1(0,y,z)=x_1^*,
\end{equation}
and 
\begin{equation}\label{derivate di x}
\partial_y x_1(0,y,z)=0, \qquad \partial_z x_1(0,y,z)=0.
\end{equation}
From \eqref{funzione implicita} we also deduce  
\begin{equation}\label{partialX1}
\partial_t x_1 (t,y,z) = - \frac{\partial_t b_3(t,x_1(t,y,z),y,z)}{\partial_{x_1}b_3(t,x_1(t,y,z),y,z)}, 
\end{equation}
\begin{equation}\label{partialX2}
\partial_y x_1 (t,y,z) = - \frac{\partial_{x_2} b_3(t,x_1(t,y,z),y,z)}{\partial_{x_1}b_3(t,x_1(t,y,z),y,z)}, 
\end{equation}
\begin{equation}\label{partialX3}
\partial_z x_1 (t,y,z) = - \frac{\partial_{x_3} b_3(t,x_1(t,y,z),y,z)}{\partial_{x_1}b_3(t,x_1(t,y,z),y,z)}.
\end{equation}
In particular, the function $x_1 (t,y,z)$ is more regular than $b_3$ (the denominator is 
non-zero from \eqref{fndjskdjgnsk2}; see also \eqref{fdhjskjdhghjs} below). 
We will need some quantitative estimates on \eqref{partialX1}-\eqref{partialX2}-\eqref{partialX3}. To do so, we start noting that from \eqref{fndjskdjgnsk2} taking the above mentioned constant 
$c >0$ sufficiently small we have in fact
\begin{equation}\label{fdhjskjdhghjs}
|\partial_{x_1}  b_3(t,x_1(t,y,z),y,z) | > \frac{M}{2}, \qquad \forall (t,y,z) \in [0,c] \times \T^2.
\end{equation} 
On the other hand, we have 
\begin{equation}\label{fhdjskjh3}
\| \partial_t  b_3 \|_{L^\infty} = \| \partial_t \rho \|_{L^\infty}  + \| \partial_t m_3 \|_{L^\infty},
\end{equation}
as well as\footnote{Here we use a crude $L^\infty$ bound since as a finer estimate is not relevant for our purposes.}
\begin{equation*}
\| \partial_t \rho \|_{L^\infty} \leq \| U \|_{L^\infty} \| \nabla \rho \|_{L^\infty} + \eta \| \Delta \rho \|_{L^\infty},
\end{equation*}
and
\begin{align*}
\| \partial_t m \|_{L^\infty} 
& \leq  \| \tilde{u} \|_{L^\infty} \| \nabla m \|_{L^\infty} 
+ \| v \|_{L^\infty} \| \nabla \tilde{b} \|_{L^\infty} + \| v \|_{L^\infty} \| \nabla m \|_{L^\infty} 
\\
 &
  + 
 \| \tilde{b} \|_{L^\infty} \| \nabla v \|_{L^\infty}+ \| m \|_{L^\infty} \| \nabla \tilde{u} \|_{L^\infty} 
 + \| m \|_{L^\infty}  \| \nabla  v \|_{L^\infty} + \eta \| \Delta{m} \|_{L^\infty}.
 \end{align*}
Thus using \eqref{stima rho 1}-\eqref{bound m tempi piccoli}-\eqref{fdhsjkjdhgfU} we get 
\begin{equation}\label{fhdjskjh1}
\| \partial_t \rho \|_{L^{\infty}} \lesssim (1 +\eta) M,
\end{equation}
\begin{equation}\label{fhdjskjh2}
\| \partial_t m \|_{L^{\infty}}
\lesssim (1 + M + \eta) \varepsilon ,
\end{equation}
for all $t \in [0, c]$, taking $c >0$ sufficiently small. \\
Taking the $\nabla$ of the equations for $\rho$ and $m$ we can prove similarly that 
\begin{equation}\label{fhdjskjh1Bis}
\| \partial_t \nabla \rho \|_{L^{\infty}} \lesssim (1 +\eta) M ,
\end{equation}
\begin{equation}\label{fhdjskjh2Bis}
\| \partial_t \nabla  m \|_{L^{\infty}} 
\lesssim (1 + M + \eta) \varepsilon ,
\end{equation}
Combining \eqref{fdhjskjdhghjs}-\eqref{fhdjskjh3}-\eqref{fhdjskjh1}-\eqref{fhdjskjh2}
we get
\begin{equation}\label{fdhjskjdhghjs1}
|\partial_t x_1(t,y,z) | \lesssim 1, \qquad \forall (t,y,z) \in [0,c] \times \T^2.
\end{equation} 
Moreover, we use 
\begin{equation}\label{fmdskjdjhgjhhhhskjldkg}
\|b(t,\cdot)\|_{H^r} \leq M + C \varepsilon,
\end{equation} 
to get that
\begin{equation}\label{fdhjskjdhghjs2}
|\partial_y x_1(t,y,z) | + |\partial_z x_1(t,y,z) | \lesssim 1, \qquad \forall (t,y,z) \in [0,c] \times \T^2.
\end{equation} 
However, the fundamental observation that we are going to make is that we can gain a 
factor $t$ in the estimates \eqref{fdhjskjdhghjs2}. To fix the ideas we will prove this for 
$\partial_y x_1 (t,y,z)$. Hereafter we will use the shortened notation  
$$
f(\star) := f(t,x_1(t,y,z),y,z)).
$$
We start noting that from \eqref{partialX2} we have
\begin{align*}
& \partial_t \partial_y x_1 (t,y,z) 
\\
& = - \frac{
\Big( \partial_t \partial_{x_2} b_3(\star ) 
+ \partial_{x_1} \partial_{x_2} b_3(\star ) \partial_t x_1(t,y,z)\Big) \partial_{x_1}b_3(\star)  -
\Big( \partial_t \partial_{x_1} b_3(\star ) 
+ \partial_{x_1} \partial_{x_2} b_3(\star ) \partial_t x_1(t,y,z) \Big) \partial_{x_2} b_3(\star)
}
{(\partial_{x_1}b_3(\star))^2}.
\end{align*}

From \eqref{fdhjskjdhghjs}-\eqref{fhdjskjh1Bis}-\eqref{fhdjskjh2Bis}-\eqref{fdhjskjdhghjs1}-\eqref{fmdskjdjhgjhhhhskjldkg}
we then deduce
$$
\sup_{t \in [0,c]} \left\| \partial_t \partial_y x_1 (t,\cdot) \right\|_{L^{\infty}_{y,z}} \lesssim 1,
$$
Combining this with \eqref{derivate di x} and the fundamental theorem of calculus  we get
\begin{equation}\label{GainT1}
\| \partial_y x_1 (t,\cdot) \|_{L^{\infty}_{y,z}} \leq \int_{0}^{t} \left\| \partial_{\tau} \partial_y x_1 (\tau,\cdot) \right\|_{L^{\infty}_{y,z}}\de \tau
\lesssim t, \qquad \forall t \in [0,c].
\end{equation}
In the exact same way we deduce 
\begin{equation}\label{GainT2}
\| \partial_z x_1 (t,\cdot) \|_{L^{\infty}_{y,z}} 
\lesssim t, \qquad \forall t \in [0,c].
\end{equation}

We now perform a fixed point argument in order to find, 
for all $t \in [0, c]$ (recall $0 < c \ll 1$), a couple $(\bar x_2,\bar x_3) \in \in (x_2^*,x_3^*) + [-\gamma,\gamma]^2$ such that
$$
m_1(t,x_1(t,\bar x_2,\bar x_3),\bar x_2,\bar x_3)=m_2(t,x_1(t,\bar x_2,\bar x_3),\bar x_2,\bar x_3)=0,
$$
where $\gamma \in (0, \pi)$ will be chosen later to be sufficiently small.
To do this, we define the vector field
\begin{equation}\label{definizione F}
F(t,y,z)=(m_1(t,x_1(t,y,z),y,z),m_2(t,x_1(t,y,z),y,z)),
\end{equation}
and we look for a fixed point of the map
\begin{equation}
\Phi^t : (y,z) \in (x_2^*,x_3^*) + [-\gamma,\gamma]^2 \to \Phi^t (y,z) := (y,z)-( J_F(0,x_2^*,x_3^*))^{-1}F(t,y,z).
\end{equation}
We recall that $(x_1^*,x_2^*,x_3^*)$ is the hyperbolic equilibrium of $b^\mathrm{in}$ and we have denoted with $J_f$ the Jacobian matrix of the function $f$ with respect to the variables $(y,z)$. 

We will show that for all $0 \leq t \leq \frac{c}{M}$ with $c >0$ sufficiently small the map $\Phi^t$ is a contraction.  
We start noting  that
\begin{equation}\label{jacobiano Phi}
J_{\Phi^t}(y,z)=\mathrm{Id}-( J_F(0,x_2^*,x_3^*))^{-1}J_F(t,y,z),
\end{equation}
with $\mathrm{Id}$ being the identity matrix of $\R^{2\times 2}$.
From \eqref{semplificazione m} we deduce 
\begin{equation}
J_F(0,x_2^*,x_3^*) =  \left(
\begin{array}{cc}
\frac{\de}{\de x_2} m_1^\mathrm{in}(x_2^*)& 0
\\
0 & \frac{\de}{\de x_3} m_2^\mathrm{in}(x_3^*)
\end{array}
\right),
\end{equation}
thus \eqref{semplificazione m 2} gives 
\begin{equation}\label{jacobiano F 0}
J_F(0,x_2^*,x_3^*)^{-1} = \e^{-1}\mathrm{Id}.
\end{equation}

Hereafter we denote with $\|(\cdot)\|_{HS}$,
$\|(\cdot)\|_{OP}$ and $\|(\cdot)\|$ the Hilbert-Schmidt, operator norm, and Euclidean norm over $\R^2$, respectively. 
We will now show that:
\begin{equation}\label{differenza jacobiani FBis}
 \|J_F(t,y,z)-J_F(0,x_2^*, x_3^*)\|_{HS} \leq \e CM t, \qquad \forall (t, y, z) \in [0, c/M] \times (x_2^*,x_3^*) + [-\gamma,\gamma]^2.
\end{equation}

To prove \eqref{differenza jacobiani FBis} we note that 
\begin{align*}
 J_F(t,y,z)-J_F(0,x_2^*, x_3^*) 
&=  \left(
\begin{array}{cc}
\partial_{x_2} m_1(\star) -  \frac{\de}{\de x_2} m_1^\mathrm{in}(x_2^*) & \partial_{x_3} m_1(\star)
\\
\partial_{x_2} m_2(\star)  & \partial_{x_3} m_2(\star)
- \frac{\de}{\de x_3} m_2(x_3^*) 
\end{array}
\right)
\\
& +
\left(
\begin{array}{cc}
\partial_{x_1} m_1(\star) \partial_y x_1(t,y,z)  & \partial_{x_1} m_1(\star) \partial_z x_1(t,y,z)
\\
\partial_{x_1} m_2(\star) \partial_y x_1(t,y,z)  & \partial_{x_1} m_2(\star) \partial_z x_1(t,y,z) 
\end{array}
\right).
\end{align*}

We first focus on the diagonal entries of the first matrix. We decompose 
\begin{align}\nonumber
\partial_{x_2}  m_1(\star)  -  \frac{\de}{\de x_2} m_1^\mathrm{in}(x_2^*) & = 
\partial_{x_2} m_1(\star) - \partial_{x_2} m_1(t,x_1(t,x_2^*,x_3^*),x_2^*,x_3^*) 
\\ \label{Combdksjhf1} & 
+ \partial_{x_2} m_1(t,x_1(t,x_2^*,x_3^*),x_2^*,x_3^*)  -  \frac{\de}{\de x_2} m_1^\mathrm{in}(x_2^*).
\end{align}
By the chain rule we have
\begin{equation*}
\partial_y\partial_{x_2} m_1(\star)  = \partial_{x_2} \partial_{x_2}  m_1(\star)+ 
 \partial_{x_1} \partial_{x_2}  m_1(\star) \partial_y x_1(t,y,z)
\end{equation*}
as well as
\begin{equation*}
\partial_z \partial_{x_2} m_1(\star)  = \partial_{x_3} \partial_{x_2}  m_1(\star)+ 
 \partial_{x_1} \partial_{x_2}  m_1(\star) \partial_z x_1(t,y,z)
\end{equation*}

Thus \eqref{bound m tempi piccoli}-\eqref{GainT1}-\eqref{GainT2} imply that 
$$ (y,z)  \in (x_2^*,x_3^*) + [-\gamma,\gamma]^2 \to \partial_{x_2}  m_1(t,x_1(t,y,z),y,z)$$ 
is 
Lipschitz 
with Lipschitz constant $C\varepsilon$ uniform over $t \in [0,c]$. Thus 
\begin{equation}\label{Combdksjhf2}
| \partial_{x_2} m_1(t,x_1(t,y,z),y,z) - \partial_{x_2} m_1(t,x_1(t,x_2^*,x_3^*),x_2^*,x_3^*) |\leq C \gamma \varepsilon.
\end{equation}

Then we observe  that 
\begin{equation*}
\frac{\de}{\de t} \partial_{x_2} m_1(\star)  = \partial_{t } \partial_{ x_2}  m_1(\star)+ 
 \partial_{x_1} \partial_{ x_2}  m_1(\star) \partial_t x_1(t,y,z)
\end{equation*}
so that
 using \eqref{fdhjskjdhghjs1}-\eqref{fhdjskjh2Bis}-\eqref{bound m tempi piccoli} we have for all $t \in [0, c/M]$: 
\begin{align}
\left\| \frac{\de}{\de t}   \partial_{x_2} m_1(\star) \right\|_{L^{\infty}_{y,z}} 
&\lesssim 
\| \partial_t \partial_{x_2} m_1(\star) \|_{L^{\infty}_{y,z}} 
+
\| \partial_{x_1} \partial_{x_2} m_1(\star) \|_{L^{\infty}_{y,z}} 
\| \partial_t x_1(t , \cdot) \|_{L^{\infty}_{y,z}} \\ \nonumber
&
\lesssim (1 + M + \eta) \varepsilon.
\end{align}
Thus using  the fundamental theorem of calculus and \eqref{fhdjskjh2Bis} we have
\begin{align}
\left| \partial_{x_2} m_1(t,x_1(t,x_2^*,x_3^*),x_2^*,x_3^*) -  \frac{\de}{\de x_2} m_1^\mathrm{in}(x_2^*) \right| 
& \leq \int_{0}^t \left| \frac{\de}{\de \tau} \partial_{x_2} m_1(\tau ,x_1(\tau ,x_2^*,x_3^*),x_2^*,x_3^*) \right|\de\tau\nonumber\\
&\lesssim (1 + M + \eta) \varepsilon t, \qquad \forall t \in [0, c/M]\label{Combdksjhf3}.
\end{align}
From \eqref{Combdksjhf1}-\eqref{Combdksjhf2}-\eqref{Combdksjhf3} we see that taking $c, \gamma >0$ sufficiently small we have 
$$
\left\| \partial_{x_2} m_1(\star) -  \frac{\de}{\de x_2} m_1^\mathrm{in}(x_2^*) \right\|_{L^{\infty}_{y,z}} \leq 
C \gamma \varepsilon  + C \varepsilon M t , 
\qquad \forall t \in [0, c/M]
$$ 
The same bound holds for the second diagonal (with exactly the same proof).\\

The off-diagonal terms are estimated as follows. 
Again we observe that, using the chain rule and \eqref{fdhjskjdhghjs1}-\eqref{fhdjskjh2Bis}-\eqref{bound m tempi piccoli}, we have 
\begin{align*}
\left\| \frac{\de}{\de t}\partial_{x_3} m_1(\star) \right\|_{L^{\infty}_{y,z}} 
&\lesssim \| \partial_{t} \partial_{x_3} m_1(\star) \|_{L^{\infty}_{y,z}}+ \| \partial_{x_1} \partial_{x_3} m_1(\star) \|_{L^{\infty}_{y,z}} \| \partial_{t} x_1(t , \cdot) \|_{L^{\infty}_{y,z}} \\
&\lesssim (1 + M + \eta) \e, \qquad\qquad \forall t \in [0, c/M].
\end{align*}
Thus, recalling \eqref{semplificazione m} we deduce that for all $\forall t \in [0, c/M]$:
\begin{equation}
\| \partial_{x_3} m_1(\star) \|_{L^{\infty}_{y,z}} \leq 
\int_{0}^t \left\| \frac{\de}{\de \tau} \partial_z m_1(\tau, x_1(\tau ,y,z),y,z) \right\|_{L^{\infty}_{y,z}}\de \tau
\lesssim (1 + M + \eta) \varepsilon t. 
\end{equation}
The same bound holds for $\partial_y m_2(t,x_1(t,y,z),y,z)$, with the same proof.

Regarding the second matrix, using \eqref{bound m tempi piccoli}-\eqref{GainT1}-\eqref{GainT2}, we can estimate all its entries by $C \varepsilon t$, as long as $t \in [0, c]$.Putting together all these bounds we deduce 
\begin{equation}\label{costante lip Phi}
\|J_{\Phi^t}\|_{HS} \leq CM t < \frac12, \qquad  \forall (t, y, z) \in [0, c/M] \times (x_2^*,x_3^*) + [-\gamma,\gamma]^2.
\end{equation}
Thus, in order to prove that $\Phi^t$ has a unique fixed point in $(x_2^*,x_3^*) + [-\gamma,\gamma]^2$, it remains to show
that 
\begin{equation}\label{prima condizione}
\Phi^t \left( (x_2^*,x_3^*) + [-\gamma,\gamma]^2 \right) \subset (x_2^*,x_3^*) + [-\gamma,\gamma]^2.
\end{equation}
To prove that, we first show that the image of $(x_2^*,x_3^*)$ stay close to this point. Indeed, we have that
\begin{align}\label{Pidfjhskdhjgsdgf}
|\Phi^t  & (x_2^*,x_3^*) - (x_2^*,x_3^*) | 
  \leq   \| (J_F(0,x_1^*,x_2^*,x_3^*))^{-1} \|_{OP} 
\| F(t,x_1(t,x_2^*,x_3^*),x_2^*,x_3^*) \| 
\\ \nonumber
&
\leq \varepsilon^{-1} \,
\|(m_1(t,x_1(t,x_2^*,x_3^*),x_2^*,x_3^*),m_2(t,x_1(t,x_2^*,x_3^*),x_2^*,x_3^*))\|.
\end{align}
On the other hand, we have (recall \eqref{Minitial})
$$
m_1 (0,x_1(0,x_2^*,x_3^*),x_2^*,x_3^*) = m^\mathrm{in}(x_1^*,x_2^*,x_3^*) = 0,   
$$
and
$$
\frac{\de}{\de t} m_1(\star) = \partial_{\tau} m_1(\star) +  \partial_{x_1} m_1(\star) \partial_\tau x_1(\tau,y,z),
$$
thus by \eqref{bound m tempi piccoli}-\eqref{fhdjskjh2}-\eqref{fdhjskjdhghjs1} and the fundamental theorem of calculus we deduce
$$
|(m_1(t,x_1(t,x_2^*,x_3^*),x_2^*,x_3^*)| \leq \int_{0}^{t} \left| \frac{\de}{\de\tau} 
(m_1(\tau,x_1(\tau,x_2^*,x_3^*),x_2^*,x_3^*)\right| \, \de \tau \leq (1 + M + \eta) \varepsilon t.
$$
The same estimate holds for $m_2$. Plugging this into \eqref{Pidfjhskdhjgsdgf} we obtain
$$
|\Phi^t   (x_2^*,x_3^*) - (x_2^*,x_3^*) |  \leq C (1 + M + \eta)  t < \frac{\gamma}{2},
$$
as long as
$$
t<\frac{\gamma}{2 C (1 + M + \eta)}.
$$
Then, using this and \eqref{costante lip Phi} we finally obtain 
\begin{align}
|\Phi^t(y,z)-(x_2^*,x_3^*)|&\leq |\Phi^t(y,z)-\Phi^t(x_2^*,x_3^*)|+|\Phi^t(x_2^*,x_3^*)-(x_2^*,x_3^*)|\nonumber\\
&\leq CM t|(y,z)-(x_2^*,x_3^*)|+\frac{\gamma}{2}\nonumber\\
&\leq CM t \gamma + \frac{\gamma}{2} \leq \gamma,
\end{align}
as long as $t \in [0, c/M]$ with $c >0$ sufficiently small. This implies \eqref{prima condizione}. 

In conclusion, by \eqref{prima condizione} and \eqref{costante lip Phi} the map $\Phi^t$ is a contraction over $(x_2^*,x_3^*) + [-\gamma,\gamma]^2$ so that for all $t \in [0, c/M]$ there exists a unique fixed point $\bar x_2(t), \bar x_3(t) \in (x_2^*,x_3^*) + [-\gamma,\gamma]^2$ of $\Phi^t$, meaning that
\begin{equation}\label{finale F}
F(t,\bar x_2(t),\bar x_3(t))=0.
\end{equation}
Now, since $\tilde b_1 = \tilde b_2 =0$, our solution has the form
$$
b(t,\cdot)= (m_1(t,\cdot),m_2(t,\cdot),b_3(t,\cdot)).
$$ 
Thus, we deduce from \eqref{funzione implicita}, \eqref{definizione F}, \eqref{finale F}
that
\begin{equation}
b(t,x_1(t,\bar x_2(t), \bar x_3(t)), \bar x_2(t), \bar x_3(t))=0,\qquad \forall t\in [0,c/M],
\qquad 0 < c \ll 1,
\end{equation}
as claimed.
In other words, the magnetic field has a zero for sufficiently small times, independently from the resistivity parameter $\eta$. 
\\
\\
\underline{\em Step 4} Large times regime.\\
\\
In this step we prove that, after some time, the magnetic field $b$ no longer possesses any equilibrium, indicating a topological change in the magnetic lines. Moreover, we show that this occurs in an accelerated regime compared to the diffusive timescale. To achieve this, it will be enough to prove the following.\\
\\
\underline{CLAIM}: 
 For every $\frac{|\ln \eta|}{\lambda(\eta)} \lesssim t < T$ we have that
\begin{equation}\label{FinalGoal}
\| b_3(t,\cdot) - \langle \tilde{b}_3 \rangle  \|_{L^\infty} <  \langle \tilde{b}_3 \rangle.
\end{equation}
\\
This claim immediately leads to the desired conclusion, and it is actually even more than what we need. Indeed, by the triangle inequality, we easily obtain
\begin{equation}
 b_3(t,x) \geq \langle \tilde{b}_3 \rangle - | b_3(t,x) - \langle \tilde{b}_3 \rangle |>0, 
\end{equation}
for times in the interval $\frac{|\ln\eta|}{\lambda(\eta)}\lesssim t<T$ and for all $x \in \T^3$. \\

We now prove the claim: we start by re-writing the solution $(u, b)$  as
\begin{equation}\label{Def:U}
b =  \tilde{b} + m, \qquad 
u =  \tilde{u} + v
\end{equation} 
where the pair $(v, m)$ satisfies the system \eqref{eq:mhdPerturbative}.
Note that, since $\langle m(t,\cdot)\rangle=0$ for all $0\leq t\leq T$, we also have
\begin{equation}\label{AverageProperty}
\langle b (t,\cdot) \rangle=\langle \tilde b (t,\cdot)\rangle=\langle\tilde{b}_3(t,\cdot) \rangle= \langle\tilde{b}^\mathrm{in}_3 \rangle.
\end{equation}
Furthermore, we recall that $\tilde{b}^\mathrm{in}_3:=M+\rho^\mathrm{in}$, with $\rho^\mathrm{in}$ belonging to the space $H_\mathrm{shear}$, and the solution $\tilde b_3$ can be written as 
\begin{equation}\label{fjdksldkfjgksldkgjksl}
\tilde{b}_3(t,x) = \langle \tilde{b}_3 \rangle + \rho(t,x),
\end{equation}
where $\rho$ solves \eqref{TransportForRho}. Since $ \rho^\mathrm{in} \in H_\mathrm{shear}$ we have the following enhanced diffusion inequality 
\begin{equation}\label{EnhDiff} 
\| \tilde{b}_3(t,\cdot) - \langle \tilde{b}_3 \rangle \|_{L^2} = \| \rho(t,\cdot)  \|_{L^2} \leq 
C e^{-\lambda(\eta) t} \| \rho^\mathrm{in} \|_{L^2},
\end{equation}
for all $t\gtrsim\frac{1}{\lambda(\eta)}$. This is still not sufficient to prove the claim, so we proceed as follows. From the definitions \eqref{Def:U}-\eqref{fjdksldkfjgksldkgjksl} we have 
$$
b_3 - \langle b_3 \rangle  =  \tilde{b}_3 + m_3 - \langle \tilde{b}_3 \rangle=\rho+m_3, 
$$ 
which implies (recall $r > 5/2$)
\begin{equation}\label{mdfksldjgnjdsldjgn}
| b_3(t,x) - \langle \tilde{b}_3 \rangle | \leq 
\|m_3(t,\cdot)\|_{L^{\infty}}  + \| \rho(t,\cdot) \|_{L^{\infty}} 
\leq \tilde C (\|m_3(t,\cdot)\|_{H^r}  + \| \rho(t,\cdot) \|_{H^r} ). 
\end{equation} 
We now estimate the two terms on the right hand side of \eqref{mdfksldjgnjdsldjgn}. First by Theorem B we obtain 
\begin{equation}\label{MainIneq}
\| \rho(t,\cdot) \|_{H^r} \leq \frac{C M}{\eta^{\frac{r}{2}}}   e^{- \lambda(\eta) t}.   
\end{equation}
Next, we define $t^*$ such that  
\begin{equation}\label{RecTime}
t^{*} = \frac{1}{\lambda(\eta)} \ln \left( \frac{2 \tilde C C}{\eta^{\frac{r}{2}}} \right),
\end{equation}
which implies that
\begin{equation}\label{definizione che mi serve}
\frac{\tilde C C }{\eta^\frac{r}{2}}  e^{-\lambda(\eta) t^*} \leq \frac12.
\end{equation}
Thus, for $t\geq t^*$, we have that
\begin{equation}\label{MainIneq3}
\tilde C \| \rho(t,\cdot) \|_{H^r} \leq   \frac{M}{2}.
\end{equation}
Moreover, by the choice of $\e$ in \eqref{scelta epsilon}, we also have from \eqref{stima m tutti i tempi} that
\begin{equation}\label{MainIneq2} 
\|m_3(t,\cdot)\|_{H^r}\lesssim  M e^{-M (T-t)}.
\end{equation}
Then, plugging \eqref{MainIneq3} and \eqref{MainIneq2} into \eqref{mdfksldjgnjdsldjgn}, the choice of $\e$ ensures that  
\begin{align}
\| b_3(t,\cdot) - \langle \tilde{b}_3 \rangle \|_{L^{\infty}} &\leq \frac{M}{2}+M e^{-M(T-t)}.\label{mfndkslkmgklsdkmgkls}
\end{align}
Thus, for $t^*\leq t<T$, by choosing $\varepsilon >0$ small enough, i.e. for $M$ large enough or alternatively $\eta$ small enough, we see that the right hand side of \eqref{mfndkslkmgklsdkmgkls} is strictly smaller than $\langle \tilde{b}_3 \rangle = M$, thereby verifying \eqref{FinalGoal}. The claim is proved. \\

We have proved that in the time interval
$$
t^* \lesssim t < T,
$$
there are no critical points of the magnetic field. We recall that, by the choice of $U$ in Step 1,
$$
\lambda(\eta) = C\,\eta^{1/2}.
$$
By using the definition of $t^*$ in \eqref{RecTime}, this yields the accelerated reconnection rate (see Definition \ref{def:fast rec})
$$
\lambda_\mathrm{rec}(\eta) := C \frac{\eta^{1/2}}{\ln \left( \frac{1}{\eta^{\frac{r}{2}}} \right) } = \frac{2C}{r}\frac{\eta^{1/2}}{|\ln \eta|},
$$
as claimed. Finally, the structural stability can be deduced from Theorem~\ref{thm:stab}, by considering sufficiently small perturbations as in \eqref{eqPertIntro2}, together with the preceding computations.
\end{proof}

\begin{rem}\label{remark dati iniziali}
In the above theorem, we have proved the existence of initial data for the \eqref{eq:mhd} equations for which the corresponding solutions exhibit magnetic reconnection on a time scale faster than the diffusive one. An important point to note is that these initial data depend on $\eta$ trough the parameter $\e$ defined in \eqref{scelta epsilon}. The reason is straightforward: in our construction, we start from globally defined solutions $(\tilde u,\tilde b)$ and we perturb them. Since the solutions are only known to exist locally in time, we must ensure that they are defined beyond the reconnection time; otherwise, our analysis would be meaningless. As a result, the existence time of the perturbed solutions depends on the size of the perturbation, which must therefore be linked to $\eta$. A more robust continuation argument could potentially allow us to consider initial data that are independent of $\eta$.
\end{rem}

\begin{rem}
We have shown that, for the solutions we constructed, the magnetic field admits an equilibrium point at least for a time interval that is independent of the resistivity~$\eta$. This scenario can therefore be extended in the limit $\eta \to 0$, where we know that the solution admits an equilibrium point for all times in its interval of existence. In particular, since the critical point is regular, it remains close to its initial position for short times, but it may drift away for large times. Nevertheless, the non-resistive solution $b_0$ satisfies the identity
$$
b_0(t, \Phi(t, x^*)) = 0,
$$
where $x^*$ is the equilibrium of the initial datum, and $\Phi$ denotes the flow of the velocity field.
Moreover, the assumptions on the data and the initial perturbation can be chosen differently. For example, $\rho^\mathrm{in}$ can depend on both variables $(x_1,x_2)$, but this will require small changes in Step 3 on small times.
\end{rem}

\begin{rem} Instead of choosing $\tilde u_3=0$, we could select $\tilde u_3$ to be any solution of the transport equation
$$
\partial_t \tilde{u}_3 + ( U \cdot \nabla ) \tilde{u}_3   = 0.
$$
The rest of the argument would proceed with only minimal adjustments, but this choice does not lead to an improvement in our result.
\end{rem}

\begin{rem} The explicit solutions constructed in Section \ref{Sec:Bifurcation} also satisfy the Hall-MHD system. In fact, it is straightforward to verify that the Hall term vanishes, i.e.,
$$
\nabla\times((\nabla\times \tilde b)\times \tilde b)=0,
$$
provided that $\tilde b$ is sufficiently regular and has the form $\tilde b=(0,0,\tilde b_3(x_1,x_2))$.
Our proof can be easily adapted to provide an example of accelerated reconnection for the Hall-MHD system. Specifically, it suffices to prove an analogous result to Theorem \ref{thm:stab} for the Hall-MHD system, which would allow us to control the $H^r$-norm of the perturbative term $m$ in \eqref{Def:U}. A similar argument applies to the construction from Section \ref{Sec:Bifurcation}, as well as to the viscous case.
\end{rem}

\begin{rem} The construction in Theorem A is highly flexible and can be adapted to other spatial domains. The core of the proof relies on the existence of a stationary solution to the 2D Euler equations which are dissipation enhancing. Specifically, we have considered the three-dimensional torus $\T^3$ as spatial domain in order to take advantage of the result in \cite[Theorem 1.1]{BCZ}. However, the proofs of Theorem \ref{thm:stab} and Theorem B do not depend critically on this particular choice of spatial domain. For instance, one could, in principle, use the results from \cite{CZDE, CZD, CLS}  to construct further examples with different reconnection times.
\end{rem}

\section{The viscous model with a stochastic force}\label{sec:stocastica}
In this section we consider a slightly different model and we show how, with the strategy build in the previous sections, it is possible to construct solutions that show {\em fast reconnection}. We consider the forced viscous model
\begin{equation}\label{eq:mhd2}\tag{F-MHD}
\begin{cases}
\partial_t u+(u\cdot \nabla)u+\nabla p= \Delta u+(b\cdot \nabla)b+f,\\
\partial_t b+(u\cdot \nabla)b=(b\cdot \nabla)u+\eta\Delta b,\\
\dive u=\dive b=0,\\
u(0,\cdot)=u^\mathrm{in},\hspace{0.3cm} b(0,\cdot)=b^\mathrm{in},
\end{cases}
\end{equation}
where $f:[0,T]\times\T^3\to\T^3$ is a given external force, while $u^\mathrm{in}$ and $b^\mathrm{in}$ are given divergence-free initial data. Our aim is to construct solutions for which the rate of reconnection $\lambda_\mathrm{rec}(\eta)$ is faster than polynomial, and in particular of logarithmic type. We recall that for the inviscid model we started from a stationary solution of the 2D Euler equations. These allowed us to define explicit $2\frac12$-D solutions of \eqref{eq:mhd}, namely solutions depending only on the first two spatial variables. This type of solutions were considered in order to decouple the equations, providing us a simple way to construct explicit solutions of \eqref{eq:mhd}. Here we adopt a similar strategy: instead of stationary 2D Euler flows, we consider solutions to the two-dimensional Navier-Stokes equations driven by an external stochastic force
\begin{equation}\label{eq:ns}\tag{NS}
\begin{cases}
\partial_t u+(u\cdot \nabla)u+\nabla p= \Delta u+f,\\
\dive u=0,\\
u(0,\cdot)=u^\mathrm{in},
\end{cases}
\end{equation}
where the unknowns are $u:[0,T]\times\T^2\to\R^2$ and $p:[0,T]\times\T^2\to\R$, while $u^\mathrm{in}:\T^2\to\R^2$ is a given divergence-free initial datum, and $f:[0,T]\times\T^2\to\R^2$ is a given body force. We now briefly recall relevant notations and results from \cite{BBP3, BBP2}, which will be instrumental in our analysis.

\subsection{Probabilistic setup}
On the space of $L^2_{\dive}:=\{v\in L^2(\T^2;\R^2):\dive v=0\}$ we define a real complete orthonormal base
\begin{equation}
e_k(x):=
\begin{cases}
\frac{1}{\sqrt{\pi}}\frac{k^\perp}{|k|}\sin(k\cdot x),\qquad k\in\Z^2_+,\\
\frac{1}{\sqrt{\pi}}\frac{k^\perp}{|k|} \cos(k\cdot x),\qquad k\in\Z^2_-,
\end{cases}
\end{equation}
where $\Z^2_+:=\{(k_1,k_2)\in\Z^2:k_2>0\}\cup\{(k_1,k_2)\in\Z^2:k_1>0, k_2=0\}$, $\Z^2_-:=-\Z^2_+$, and $\Z^2_-\cup \Z^2_+=\Z^2_*:=\Z^2\setminus\{(0,0)\}$. Define the functional space
$$
\mathbf{L}^2:=\left\{g\in L^2(\T^2;\R^2):\int_{\T^2}g(x)\de x=0,\quad \dive g=0\right\},
$$
and let $W_t$ be a cylindrical Wiener process on $\mathbf{L}^2$ defined by
$$
W_t=\sum_{k\in\Z^2_*}e_k W_t^k,
$$
where $\{W_t^k\}_{k\in\Z^2_*}$ are a collection of i.i.d. one-dimensional Wiener processes on a common canonical filtered probability space $(\Omega,\mathcal{F},\mathcal{F}_t,\mathbb{P})$. Note that the definition of $e_k$ implies that $W_t$ is divergence-free.
We denote by $Q$ a positive Hilbert-Schmidt operator on $\mathbf{L}^2$ which can be diagonalized with respect to $\{e_k\}_{k\in\Z^2_*}$ with eigenvalues $\{q_k\}\in~\ell^2(\Z^2_*)$ defined by
$$
Qe_k=q_ke_k,\qquad k\in\Z^2_*.
$$
We assume that the force is of the type
\begin{equation}\label{forza}
f=Q\dot{W}_t,
\end{equation}
which takes the form 
$$
Q\dot{W}_t(x)=\sum_{k\in\Z^2_*}q_k e_k (x)\dot{W}_t^k,
$$
for any $t>0$ and $x\in\T^2$. On $Q$ we assume the following regularity and non-degeneracy assumption.
\begin{assumption}
There exists $\alpha$ satisfying $\alpha>\frac{5d}{2}$ and a constant C such that
\begin{equation}\label{smoothing effect Q}
\frac{1}{C}\|(-\Delta)^{-\alpha/2}u\|_{L^2}\leq \|Qu\|_{L^2}\leq C\|(-\Delta)^{-\alpha/2}u\|_{L^2}
\end{equation}
Equivalently,
\begin{equation}\label{coloring assumption}
|q_k|\approx |k|^{-\alpha}
\end{equation}
\end{assumption}
The assumption above basically says that the forcing term has high spatial regularity but cannot be $C^\infty$.
We also denote with $\mathbf{H}^s$ the space
$$
\mathbf{H}^s:=\left\{g\in H^s(\T^2;\R^2):\int_{\T^2}g(x)\de x=0,\quad \dive g=0\right\},
$$
and $\mathbf{H}:=H^\sigma$ for some fixed $\sigma \in (\alpha-2,\alpha-1)$ with $\alpha$ big enough to have the embedding $\mathbf{H}\hookrightarrow C^3$. Note that assumption \eqref{coloring assumption} and the definition of $\sigma$ imply that $\{|k|^\sigma q_k\}$ is square summable over $\Z^2_*$ and therefore $QW_t\in \mathbf{H}$ almost surely.\\
 
The following result can be found in \cite{Kuksin book, BBP}.
\begin{prop}[Proposition A.3 in \cite{BBP}]\label{propmu}
For $\mathbb{P}$-almost every $\omega\in \Omega$, for all $u^\mathrm{in}\in\mathbf{H}\cap H^\gamma(\T^2)$ with $\gamma<\alpha-1$ and for all $T>0$, $1\leq p<\infty$, we have that there exists a unique solution $u\in C([0,T];\mathbf{H}\cap H^\gamma(\T^2))$ of \eqref{eq:ns} with $u(0,\cdot)=u^\mathrm{in}$. Moreover, the process is $\mathcal{F}_t$-adapted with $u\in L^p(\Omega;C([0,T];\mathbf{H}))\cap L^2(\Omega;L^2([0,T];H^{\gamma+1}(\T^2)))$. Additionally, for all $p\in[1,\infty)$ and for all $\delta>0$ such that $\gamma+\delta<\alpha-1$, the following estimates hold
\begin{align}
\mathbb{E}\left[\sup_{t\in[0,T]}\|u(t,\cdot)\|_{H^\gamma}^p\right]&\leq C(T,p,\gamma)(1+\|u^\mathrm{in}\|_{\mathbf{H}\cap H^\gamma}^p),\\
\mathbb{E}\left[\int_0^T\|u(t,\cdot)\|_{H^{\gamma+1}}^2\de t\right]&\leq C(T,\delta) (1+\|u^\mathrm{in}\|_{H^\gamma}).
\end{align}
\end{prop}

Adapting the argument of \cite[Proposition 2.4.12]{Kuksin book}, we prove the energy estimates in Proposition \ref{propmu}, with particular emphasis on the dependence of the constants on $T$ as this will play a crucial role in the subsequent analysis. We have the following.
\begin{prop}\label{propmu1}
Let $u$ be the unique solution of the equations \eqref{eq:ns} provided by Proposition \ref{propmu} and denote by 
\begin{equation}\label{def:Cgamma}
C_\ell=\sum_{k\in\Z^2_*}|k|^{2\ell}q_k^2,
\end{equation}
a positive constant which is finite as long as $\ell<\alpha-1$. Then, for all $\ell<\alpha-1$ we have that
\begin{equation}
\mathbb{E}\left[\sup_{t\in[0,T]}\|u(t,\cdot)\|_{H^\ell}^2+\int_0^T\|u(t)\|_{H^{\ell+1}}^2\de t\right]\leq C T^{2\ell+\frac32},
\end{equation}
for some constant $C$ depending on $\ell,C_0, C_\ell, \|u^\mathrm{in}\|_{H^\ell}$.
\end{prop}
\begin{proof}

Let $u$ be the solution of \eqref{eq:ns}, which we rewrite as
$$
\begin{cases}
\de u + \big( A u + B(u,u)\big)\,\de t = Q\,\de W_t,\\
u(0)=u^\mathrm{in},
\end{cases}
$$
where $A=-\mathbb{P}\Delta$ is the Stokes operator and $B(u,u)=\mathbb{P}(u\cdot\nabla u)$. 
We start by proving the $L^2$ and $H^1$ estimates, for which we a linear growth in time.\\
\\
\underline{\bf $L^2$ estimate}\\
\\
We apply the Itô's formula to $f(u)=\|u\|_{L^2}$ and we get that
\begin{equation}\label{ito l2}
\de \|u\|_{L^2}^2 + 2\|\nabla u(t,\cdot)\|_{L^2}^2\,\de t =C_0 \de t + 2\langle u(t,\cdot),\,Q\,\de W_t\rangle.
\end{equation}
We now define the stopping time
\begin{equation}
\tau_R:=\inf\{t\geq 0: \|u(t,\cdot)\|_{L^2}\geq R\}.
\end{equation}
We integrate \eqref{dopo ito 2} with respect to time, we take the supremum over $[0,T\wedge\tau_R]$ and we get that
\begin{align*}
\sup_{0\leq t\leq T\wedge\tau_R}\|u(t,\cdot)\|_{L^2}^2+2\int_0^{T\wedge\tau_R}\|\nabla u(s,\cdot)\|_{L^2}^2\de s&\leq \|u^\mathrm{in}\|_{L^2}^2+C_0 (T\wedge \tau_R)\\
&+2\sup_{0\leq t\leq T\wedge\tau_R}\int_0^{t}\langle u(s,\cdot),\,Q\,\de W_s\rangle
\end{align*}
We now estimate the martingale term
\begin{equation}
M_{T\wedge\tau_R}:=\sup_{0\leq t\leq T\wedge\tau_R}\int_0^{t}\langle u(s,\cdot),\,Q\,\de W_s\rangle=\sup_{0\leq t\leq T\wedge\tau_R} Z_t.
\end{equation}
Since the quadratic variation $\langle Z\rangle_t$ can be bounded by
\begin{equation}
\langle Z\rangle_t=\int_0^t\sum_k\langle u(s,\cdot),Qe_k\rangle^2 \de s\leq C_0\int_0^t\|u(s,\cdot)\|_{L^2}^2\de s,
\end{equation}
by the Burkholder-Davis-Gundy inequality, H\"older and Young inequality we have that
\begin{align*}
\E\left[2\sup_{0\leq t\leq T\wedge\tau_R}\int_0^{t}\langle u(t,\cdot),\,Q\,\de W_t\rangle|\right]&\leq C\E\left[\left(C_0\int_0^{T\wedge\tau_R}\|u(s,\cdot)\|_{L^2}^2\de s\right)^\frac12\right]\\
&\leq C\sqrt{C_0 T}\, \E\left[\sup_{0\leq t\leq T\wedge\tau_R}\|u(t,\cdot)\|_{L^2}^2\right]^\frac12\\
&\leq C C_0 T+\frac12 \E\left[\sup_{0\leq t\leq T\wedge\tau_R}\|u(t,\cdot)\|_{L^2}^2\right].
\end{align*}
Then, we take the expectation in \eqref{stop 1} and we use the bound above to get that
\begin{equation}
\E\left[\sup_{0\leq t\leq T\wedge\tau_R}\|u(t,\cdot)\|_{L^2}^2\right]+4\E\left[\int_0^{T\wedge\tau_R}\|\nabla u(s,\cdot)\|_{L^2}^2\de s\right]\leq 2\|u^\mathrm{in}\|_{L^2}^2+2C_0 T.
\end{equation}
Since the right hand side does not depend on $R$, by monotone convergence theorem we get that
\begin{equation}\label{fine step 0}
\E\left[\sup_{0\leq t\leq T}\|u(t,\cdot)\|_{L^2}^2\right]+\E\left[\int_0^{T}\|\nabla u(s,\cdot)\|_{L^2}^2\de s\right]\leq 2(\|u^\mathrm{in}\|_{L^2}^2+C_0 T).
\end{equation}
\\
\underline{\bf The $H^1$ bound}\\
\\
We now apply Itô's formula to $f(u)=\|u\|_{H^1}^2=\|A^{1/2}u\|_{L^2}^2=\langle Au,u\rangle$:
\begin{equation}\label{dopo ito}\de \|u\|_{H^1}^2 + 2\|\nabla u(t)\|_{H^1}^2\,\de t =- 2\langle \Delta u(t),\,B(u(t),u(t))\rangle\,\de t+ C_1 \de t + 2\langle \Delta u(t),\,Q\,\de W_t\rangle.
\end{equation}
Note that, also in this case, there is no contribution of the nonlinear term thanks to the $2D$ structure: indeed, if we define the stream function $\psi$ of $u$, we have that
$$
u=\curl \psi,\qquad \curl u=\Delta \psi,\qquad \curl\left(u\cdot\nabla u\right)=u\cdot\nabla \curl u,
$$
and
\begin{align*}
\langle \Delta u(t),\,B(u(t),u(t))\rangle&=\int_{\T^2}u\cdot\nabla u \,\Delta u\,\de x=\int_{\T^2}u\cdot\nabla u \,\Delta \curl\psi\,\de x\\
&=\int_{\T^2}u\cdot\nabla \curl u \,\Delta \psi\,\de x=\int_{\T^2}u\cdot\nabla |\Delta \psi|^2\,\de x=0
\end{align*}
We define the stopping time
\begin{equation}
\tau_R:=\inf\{t\geq 0: \|u(t,\cdot)\|_{H^1}\geq R\},
\end{equation}
we integrate in time, we take the supremum $[0,T\wedge\tau_R]$ to obtain that 
\begin{align*}
\sup_{0\leq t\leq T\wedge\tau_R}\|u(t,\cdot)\|_{H^1}^2+2\int_0^{T\wedge\tau_R}\|\nabla u(s,\cdot)\|_{H^1}^2\de s&\leq \|u^\mathrm{in}\|_{H^1}^2+C_1 T\\
&+2\sup_{0\leq t\leq T\wedge\tau_R}\int_0^{t}\langle \Delta u(s,\cdot),\,Q\,\de W_s\rangle.
\end{align*}
We now deal with the martingale term: by proceeding similarly to the $L^2$ case, we use \eqref{smoothing effect Q} and \eqref{fine step 0} to obtain that
\begin{align*}
\E\left[2\sup_{0\leq t\leq T\wedge\tau_R}\int_0^{t}|\langle \Delta u(s,\cdot),\,Q\,\de W_s\rangle|\right]&\leq C\E\left[\left(C_1\int_0^{T\wedge\tau_R}\| u(s,\cdot)\|_{H^1}^2\de s\right)^\frac12\right]\\
&\leq C\sqrt{C_1T}\, \E\left[\sup_{0\leq s\leq T\wedge\tau_R}\|u(s,\cdot)\|_{H^1}^2\right]^\frac12\\
&\leq CC_1T+\frac12\E\left[\sup_{0\leq s\leq T\wedge\tau_R}\|u(s,\cdot)\|_{H^1}^2\right].
\end{align*}
So we have obtained that
\begin{align}
\E\left[\sup_{0\leq t\leq T\wedge\tau_R}\| u(t,\cdot)\|_{H^1}^2\right]+\E\left[\int_0^{T\wedge\tau_R}\|\nabla u(s,\cdot)\|_{H^1}^2\de s\right]&\leq 2\|u^\mathrm{in}\|_{H^1}^2+2C C_1T,
\end{align}
and the conclusion follows applying the monotone convergence theorem.\\
\\
\underline{\bf Higher Sobolev norms}\\
\\
We now assume that $\ell\geq 2$ and the estimate holds for $\ell-1$. Differently from the previous case, we have to handle the contribution of the nonlinear term in the estimates. First, we apply Itô's formula to $f(u)=\|u\|_{H^\ell}^2=\|A^{\ell/2}u\|_{L^2}^2=\langle A^\ell u, u\rangle$: 
\begin{equation}\label{dopo ito}
\de \|u\|_{H^\ell}^2 + 2\|u(t)\|_{H^{\ell+1}}^2\,\de t 
=- 2\langle A^{\ell}u(t),\,B(u(t),u(t))\rangle\,\de t
+ C_\ell \de t + 2\langle A^{\ell}u(t),\,Q\,\de W_t\rangle.
\end{equation}
where we used \eqref{def:Cgamma}. In order to estimate the nonlinear term, we use the following interpolation estimate
\begin{align*}
|\langle A^{\ell}u(t),\,B(u(t),u(t))\rangle|&\leq \|u(t)\|_{H^{\ell+1}}^{\frac{4\ell-1}{2\ell}}\|u(t)\|_{H^1}^{\frac{\ell+1}{2\ell}}\|u(t)\|_{L^2}^\frac12\\
&\leq \|u(t)\|_{H^{\ell+1}}^2+C(\ell)\|u(t)\|_{H^1}^{2(\ell+1)}\|u(t)\|_{L^2}^{2\ell}.
\end{align*}
Thus, we can rewrite \eqref{dopo ito} as
\begin{equation}\label{dopo ito 2}
\de \|u\|_{H^\ell}^2 + \|u(t)\|_{H^{\ell+1}}^2\,\de t \leq C(\ell)\|u(t)\|_{H^1}^{2(\ell+1)}\|u(t)\|_{L^2}^{2\ell}\,\de t + C_\ell \de t + 2\langle A^{\ell}u(t),\,Q\,\de W_t\rangle.
\end{equation}
We now define the stopping time
\begin{equation}
\tau_R:=\inf\{t\geq 0: \|u(t)\|_{H^\ell}\geq R\}.
\end{equation}
We integrate \eqref{dopo ito 2} in time, we take the supremum over the interval $[0,T\wedge\tau_R]$, and we get that
\begin{align}
\sup_{0\leq t\leq T\wedge\tau_R}\|u(t,\cdot)\|_{H^\ell}^2+\int_0^{T\wedge\tau_R}\|u(s,\cdot)\|_{H^{\ell+1}}^2\de s&\leq \|u^\mathrm{in}\|_{H^\ell}^2+C_\ell T\nonumber\\
&+C(\ell)\int_0^{T\wedge\tau_R} \|u(s,\cdot)\|_{H^1}^{2(\ell+1)}\|u(s,\cdot)\|_{L^2}^{2\ell}\de s\nonumber\\
&+2\sup_{0\leq t\leq T\wedge\tau_R}\int_0^{t}\langle A^{\ell}u(s,\cdot),\,Q\,\de W_s\rangle\label{stop 1}.
\end{align}
We now estimate the martingale term
\begin{equation}
M_T:=\sup_{0\leq t\leq T\wedge\tau_R}\int_0^{t}\langle A^{\ell}u(s,\cdot),\,Q\,\de W_s\rangle=\sup_{0\leq t\leq T\wedge\tau_R} Z_t.
\end{equation}
Since the quadratic variation $\langle Z\rangle_t$ can be bounded as
\begin{equation}
\langle Z\rangle_t=\int_0^t\sum_k\langle A^{\ell}u(s,\cdot),Qe_k\rangle^2 \de s\leq C_\ell\int_0^t\|u(s,\cdot)\|_{H^\ell}^2\de s,
\end{equation}
by the Burkholder-Davis-Gundy inequality, H\"older and Young inequality we have that
\begin{align*}
\E\left[2\sup_{0\leq t\leq T\wedge\tau_R}\int_0^{t}|\langle A^{\ell}u(s,\cdot),\,Q\,\de W_s\rangle|\right]&\leq C\E\left[\left(C_\ell\int_0^{T\wedge\tau_R}\|u(s,\cdot)\|_{H^\ell}^2\de s\right)^\frac12\right]\\
&\leq C\sqrt{C_\ell T}\, \E\left[\sup_{0\leq t\leq T\wedge\tau_R}\|u(t,\cdot)\|_{H^\ell}^2\right]^\frac12\\
&\leq C C_\ell T+\frac12 \sup_{0\leq t\leq T\wedge\tau_R}\|u(t,\cdot)\|_{H^\ell}^2.
\end{align*}
Thus, we take the expectation in \eqref{stop 1} and we use the bound above to get that
\begin{align}
\E\left[\sup_{0\leq t\leq T\wedge\tau_R}\|u(t,\cdot)\|_{H^\ell}^2\right]+4\E\left[\int_0^{T\wedge\tau_R}\|u(s,\cdot)\|_{H^{\ell+1}}^2\de s\right]&\leq 2\|u^\mathrm{in}\|_{H^\ell}^2+2C_\ell T\nonumber\\
+2\E\left[\int_0^{T\wedge\tau_R}\right.&\left. \|u(s,\cdot)\|_{H^1}^{2(\ell+1)}\|u(s,\cdot)\|_{L^2}^{2\ell}\de s\right].\label{prima del nonlineare}
\end{align}
In order to close the estimate, we have to deal with the last term on the right hand side above. We need the following momentum estimates, see \cite[Corollary 2.4.11]{Kuksin book}.
\begin{cor}
Define 
$$
\mathcal{E}_u(t)=\|u(t)\|_{L^2}^2+\int_0^t\|\nabla u(s)\|_{L^2}^2\de s.
$$
assume that the random variable $u^\mathrm{in}$ satisfies the condition
$$
\E\|u^\mathrm{in}\|_{L^2}^{2m}<\infty,
$$
where $m\geq 1$ is an integer. Then there is a constant $C(m)>0$ depending only on $m$ such that for any $T\geq 1$ we have
\begin{equation}
\E\left[\sup_{0\leq t\leq T}\mathcal{E}_u(t)^m\right]\leq C(m)\left(\E\|u^\mathrm{in}\|_{L^2}^{2m}+ T^m + \gamma^{-m}\right),
\end{equation}
where $\gamma>0$ is defined as $\gamma=\frac14 (\sup_k q_k^2)^{-1}$.
\end{cor}
Since we are considering deterministic initial datum, we can apply the above estimates with $m=2\ell$ and $m=2(\ell+1)$ to obtain 
\begin{align}
\E\left[\int_0^{T\wedge\tau_R} \|u(s,\cdot)\|_{H^1}^{2(\ell+1)}\|u(s,\cdot)\|_{L^2}^{2\ell}\de s\right]&\leq \E\left[\sup_{0\leq t\leq T}\|u(t,\cdot)\|_{L^2}^{2\ell}\int_0^{T}\|u(s,\cdot)\|_{H^1}^{2(\ell+1)}\de s\right]\nonumber\\
&\leq \E\left[\sup_{0\leq t\leq T}\|u(t,\cdot)\|_{L^2}^{4\ell}\right]^\frac12\E\left[\left(\int_0^{T} \|u(s,\cdot)\|_{H^1}^{2(\ell+1)}\de s\right)^2\right]^\frac12\nonumber\\
&\leq \sqrt{T}\,\E\left[\sup_{0\leq t\leq T}\|u(t,\cdot)\|_{L^2}^{4\ell}\right]^\frac12\E\left[\int_0^{t} \|u(s,\cdot)\|_{H^1}^{4(\ell+1)}\de s\right]^\frac12\nonumber\\
&\leq C(\ell, C_0) T^{2\ell+\frac32}.\label{fine non lineare}
\end{align}
We put \eqref{fine non lineare} into \eqref{prima del nonlineare} and using monotone convergence theorem, we finally get
\begin{equation}
\E\left[\sup_{0\leq t\leq T}\|u(t,\cdot)\|_{H^\ell}^2\right]+\E\left[\int_0^{T}\|u(s,\cdot)\|_{H^{\ell+1}}^2\de s\right]\leq 2\|u^\mathrm{in}\|_{H^\ell}^2+2C_\ell T+2C(\ell, C_0) T^{2\ell+\frac32},
\end{equation}
and this concludes the proof.
\end{proof}

We recall the following theorem, see \cite[Theorem 1.3]{BBP2}.
\begin{thm}\label{teo:bedrossian}
For all $s\geq 0$ and all $p\in[1,\infty)$, there exists a deterministic $\hat{\gamma} =\hat{\gamma}(s,p)$ (depending only on $s,p$ and the parameters $Q,\nu$) which satisfies the following properties. For all $\eta\in [0,1]$, and for all $u^\mathrm{in}\in\mathbf{H}$
there is a $\mathbb{P}$-a.s. finite random constant $D_\eta(\omega,u^\mathrm{in}):\Omega\times\mathbf{H}\mapsto[1,\infty)$ (also depending
on $p,s$) such that the solution to \eqref{eq:ad} with $U$ being the solution of \eqref{eq:ns} with initial data $u^\mathrm{in}$, satisfies for all $\rho^\mathrm{in}\in H^s(\T^2)$ (mean-zero),
\begin{equation}\label{eq:exp mix}
\|\rho(t,\cdot)\|_{H^{-s}}\leq D_\eta(\omega,u^\mathrm{in})e^{\hat{\gamma}t}\|\rho^\mathrm{in}\|_{H^s},
\end{equation}
where $D_\eta(\omega,u^\mathrm{in})$ satisfies the following $\eta$-independent bound: there exists a $\beta\geq 2$ (independent of $u^\mathrm{in}, p, s$) such that for all $\kappa>0$,
\begin{equation}\label{aggiusto 2}
\mathbb{E}[D_\eta(\cdot,\rho^\mathrm{in})^p]\lesssim_{\kappa,p}(1+\|u^\mathrm{in}\|_{\mathbf{H}})^{p\beta}\exp\left(\kappa\|\curl u^\mathrm{in}\|_{L^2}\right)^2.
\end{equation}
\end{thm}
In \cite[Theorem 1.4]{BBP3} it was proved the following enhanced dissipation estimate.
\begin{thm}\label{teo:bedrossian enhanced}
In the setting of Theorem \eqref{teo:bedrossian}, for any $p\geq 2$, let $\hat{\gamma}(1,p)$ as in Theorem \eqref{teo:bedrossian}. For all $\eta\in(0,1]$, and for all $u^\mathrm{in}\in\mathbf{H}$ there is a $\mathbb{P}$-a.s. finite random constant $D_\eta'(\omega,\rho^\mathrm{in}):\Omega\times\mathbf{H}\mapsto[1,\infty)$ (also depending on $p$) such that the solution to \eqref{eq:ad}, satisfies for all $\rho^\mathrm{in}\in H^s(\T^d)$ (mean-zero) and $u^\mathrm{in}\in\mathbf{H}$,
\begin{equation}\label{eq:stima l2 bedrossian}
\|\rho(t,\cdot)\|_{L^2}\leq D_\eta'(\omega,u^\mathrm{in})\eta^{-1}e^{-\hat{\gamma} t}\|\rho^\mathrm{in}\|_{L^2},
\end{equation}
where $D_\eta'(\omega,\rho^\mathrm{in})$ also satisfies the following $\eta$-independent bound for $\beta$ sufficiently large (independent of $u,p,\eta$) and for all $\kappa>0$,
\begin{equation}
\mathbb{E}[D_\eta'(\cdot,\rho^\mathrm{in})^p]\lesssim_{\kappa,p}(1+\|u^\mathrm{in}\|_{\mathbf{H}})^{p\beta}\exp\left(\kappa\|\curl u^\mathrm{in}\|_{L^2}\right)^2.
\end{equation}
\end{thm}

\subsection{Fast reconnection}
We now have all the ingredients in order to prove our last theorem.

\begin{mainthm4}[Fast reconnection]
Let $r > 5/2$. There exists $\eta_0 >0$ sufficiently small such that there exists a family of divergence-free vector fields $(u^\mathrm{in}_\eta,b^\mathrm{in}_\eta)$ in $H^r(\T^3)$ with $0<\eta<\eta_0$, and a stochastic force $f\in L^\infty((0,T);H^r(\T^3))$, defined on some given filtered probability space $(\Omega,\mathcal{F},\mathcal{F}_t,\mathbb{P})$, such that the following holds: for any $\delta>0$ there exists a subset $\bar \Omega\subset \Omega$ with $\mathbb{P}(\Omega\setminus\bar\Omega)\leq \delta$ such that the unique local solution $(u_{\eta},b_{\eta},p_{\eta})$ of \eqref{eq:mhd2} arising from $(u^\mathrm{in}_\eta,b^\mathrm{in}_\eta)$ shows fast magnetic reconnection with rate 
$$
\lambda_\mathrm{rec}(\eta) := \frac{C}{|\ln \eta|},
$$
for some positive deterministic constant $C>0$, and for all realizations $\omega\in\bar\Omega$. More precisely, the magnetic field $b_{\eta}(t,\cdot)$ has an equilibrium point for times $t \in [0, c_1]$ and has no equilibrium points for times 
$$
t \in \left[C_\delta |\ln \eta|, 2 C_\delta|\ln \eta|\right].
$$  
The statement is structurally stable, in the sense that it still holds 
for all solution with (divergence-free) initial data $(\tilde u^\mathrm{in}_\eta, \tilde b^\mathrm{in}_\eta)$
such that 
\begin{equation}\label{eqPertIntroSection}
\| \tilde u^\mathrm{in}_\eta - u^\mathrm{in}_\eta\|_{H^r} 
 + \|\tilde b^\mathrm{in}_\eta -  b^\mathrm{in}_\eta \|_{H^r} \leq c(r, \eta)
\end{equation}
with $c(r, \eta)$ sufficiently small.
\end{mainthm4}
\begin{proof}
We divide the proof in several steps. Since the arguments are analogous to those in the proof of Theorem~A, we omit some details.\\
\\
\underline{Step 1} Reference solution.\\
\\
The crucial difference with respect to the proof of Theorem A is the choice of the reference solution we made in the Step 1. Specifically, we consider a solution of the form
\begin{equation}
\tilde u=(U_1,U_2, U_3),\qquad \tilde b=(0,0,\tilde b_3),
\end{equation}
where 
\begin{itemize}
\item $U=(U_1(t,x_1,x_2),U_2(t,x_1,x_2)):[0,T]\times\T^2\to\R^2$ is the $\mathbb{P}$-a.s. unique, global-in-time solution of the forced Navier-Stokes equation \eqref{eq:ns} (provided by Proposition \ref{propmu}) with forcing term $(f_1,f_2)$ as in \eqref{forza} and initial datum a divergence-free vector field $U^\mathrm{in}:\T^2\to\R^2$,
\item $U_3=U_3(t,x_1,x_2)$ is the unique solution of the linear problem
\begin{equation}
\begin{cases}
\partial_t U_3+U\cdot \nabla U_3=\Delta U_3+f_3,\\
U_3(0,\cdot)=U^\mathrm{in}_3,
\end{cases}
\end{equation}
where $U^\mathrm{in}_3:\T^2\to\R$ and $f_3:\T^3\to\R$ are smooth functions. Note that, since the first two components of $\tilde u$ depend only on $(x_1,x_2)$, the divergence-free condition on $\tilde u$ imposes that also $U_3$ is a function of $(x_1,x_2)$. Indeed, the pressure is obtained as usual by solving the elliptic equation 
$$
-\Delta p=\dive(U\cdot\nabla U),
$$
and thus it does not depend on $x_3$.
\item $\tilde b_3:[0,T]\times\T^2\to \R$ is the solution of the advection-diffusion equation
\begin{equation}\label{richiamo finale}
\begin{cases}
\partial_t \tilde b_3+U\cdot\nabla \tilde b_3=\eta\Delta\tilde b_3,\\
\tilde b_3(0,\cdot)=\tilde{b}^\mathrm{in}_3,
\end{cases}
\end{equation}
where the function $\tilde{b}^\mathrm{in}_3:\T^2\to\R$ is chosen so that:
\begin{itemize}
\item there exists $(x_1^*,x_2^*)\in \T^2$ such that $\tilde{b}^\mathrm{in}_3(x_1^*,x_2^*)=0$ and $\partial_{x_1}\tilde{b}^\mathrm{in}_3(x_1^*,x_2^*)\simeq M$;
\item it has strictly positive average, i.e. for some given $M>0$
\begin{equation}\label{fkaldkfalkfmkasl}
\langle \tilde{b}^\mathrm{in}_3 \rangle = M,\quad \mbox{and }\quad\| \tilde{b}^\mathrm{in}_3\|_{H^r} \simeq M;
\end{equation}
\item it can be written as $\tilde{b}^\mathrm{in}_3:=M+\rho^\mathrm{in}$, with $\rho^\mathrm{in}:\T^2\to\R$ having zero average.
\end{itemize}
\end{itemize}
Note that $(\tilde u,\tilde b)$ is a global-in-time solution of \eqref{eq:mhd2}. We then consider the initial datum
\begin{align*}
b^\mathrm{in}(x_1, x_2, x_3) &= \tilde{b}^\mathrm{in}(x_1, x_2) +  m^\mathrm{in} (x_1, x_2, x_3),\\
u^\mathrm{in} &= \tilde{u}^\mathrm{in},
\end{align*}
where $m^\mathrm{in} (x_1, x_2, x_3)$ is a zero average smooth divergence-free vector field with a regular equilibrium point at $(x_1^*,x_2^*,x_3^*)\in \T^3 $ and such that 
$$
\| m^\mathrm{in} \|_{H^r} = \varepsilon \ll 1,
$$ 
for some $r > 5/2$. As for Theorem A, an example of initial datum is given by considering
$$
\tilde b^\mathrm{in}_3=M-2M\cos \left(x_1 -\frac{\pi}{3}- x_1^*\right),\qquad m^\mathrm{in}(x_1,x_2,x_3)=\e(\sin (x_2- x_2^*), \sin (x_3 - x_3^*), 0).
$$
Without loss of generality we can also assume that $\|u^\mathrm{in}\|_{H^r}\simeq 1$.\\
\\
\underline{Step 2} Perturbative argument\\
\\
We write the solution of \eqref{eq:mhd2} as
$$
b=\tilde b+m,\quad u=\tilde u+v,
$$
where the pair $(v, m)$ satisfies the system 
\begin{equation}\label{eq:mhdPerturbative2}
\begin{cases}
\partial_t v + (\tilde{u} \cdot \nabla)v  + (v\cdot \nabla)\tilde{u} + (v\cdot \nabla)v+\nabla P_v=\Delta v+
(\tilde{b} \cdot \nabla)m + ( m \cdot \nabla)\tilde{b} + (m \cdot \nabla)m,\\
\partial_t m + (\tilde{u} \cdot \nabla) m + (v \cdot \nabla) \tilde{b} + (v \cdot \nabla) m   = 
 (\tilde{b} \cdot \nabla) v + (m \cdot \nabla) \tilde{u} + (m \cdot \nabla) v + \eta \Delta{m},\\
\dive v=\dive m=0,\\
v(0,\cdot)=0,\hspace{0.3cm} m(0,\cdot)=m^\mathrm{in}.
\end{cases}
\end{equation}
Notice that there is no force in \eqref{eq:mhdPerturbative2} and, moreover, it is a system of determinist PDEs with random coefficients $(\tilde u,\tilde b)$.
As observed in Remark \ref{rem:stabilità viscoso}, the proof of Theorem \ref{thm:stab} in the inviscid framework carries over to the viscous case, so that the same stability result applies to \eqref{eq:mhd2}. The only point to be careful about is that all the estimates hold almost surely, as the coefficients $\tilde u,\tilde b$ in the system \eqref{eq:mhdPerturbative2} are random. In order to avoid dealing explicitly with random coefficients and to keep all constants deterministic, we select a full-measure set $\bar \Omega\subset\Omega$ such that all the estimates are valid uniformly on $\bar \Omega$. First, we define the set
\begin{equation}\label{restrizione 1}
\Omega_1=\left\{\omega\in\Omega: \sup_{t\in[0,T]}\|\tilde u(t,\cdot)\|_{H^r}^2+ \int_0^T\|\tilde u(\tau,\cdot)\|_{H^{r+1}}^2\de\tau>R\right\}.
\end{equation}
By Chebyshev inequality and the bounds in Proposition \ref{propmu1} we have that
\begin{equation}\label{scelta R}
\mathbb{P}(\Omega_1)\leq \frac{CT^{2r+\frac32}}{R},
\end{equation}
which can be made as small as we want by taking $R\gg 1$. Moreover, if consider the set
$$
\Omega_2:=\{\omega\in\Omega: D_\eta'(\omega,u^\mathrm{in})>R\},
$$
by \eqref{aggiusto 2} in Theorem \ref{teo:bedrossian enhanced} we also have that
$$
\mathbb{P}(\Omega_2)\leq \frac{C}{R}.
$$
for some positive constant $C$ which does not depend on $\eta$. Thus, we define the set
$$
\bar\Omega:=\Omega\setminus(\Omega_1\cup \Omega_2),
$$
and in what follows, we restrict to a fixed realization $\omega\in\bar\Omega$, and all constants will be understood as deterministic. With this convention, the short-time analysis proceeds exactly as in the proof of Theorem~A. We now turn to the study of the reconnection time.
By following the proof of Theorem A, the crucial part is to prove that for every $\frac{1}{\lambda_\mathrm{rec}(\eta)} \lesssim t < T$, we have that
\begin{equation}\label{bhbhbhbhbhbh}
\| b_3(t,\cdot) - \langle \tilde{b}_3 \rangle  \|_{L^\infty} <  \langle \tilde{b}_3 \rangle.
\end{equation}
Below we will consider $\e$ small enough so that the solution $(u,b)$ is defined up to a given time $T\sim \frac{1}{\eta}$, i.e. greater than the reconnection time.\\
\\
\underline{Step 3} Size of the perturbation\\
\\
To prove \eqref{bhbhbhbhbhbh}, we decompose $b_3=M+\rho$, where the zero-average part $\rho$ is the solution of
\begin{equation}\label{eq:ad dimostrazione stocastica}
\begin{cases}
\partial_t\rho+U\cdot\nabla\rho=\eta\Delta\rho,\\
\rho(0,\cdot)=\rho^{\mathrm{in}},
\end{cases}
\end{equation}
with $\rho^{\mathrm{in}}$ being the zero-average part of $\tilde b_3^\mathrm{in}$, with $\|\rho^\mathrm{in}\|_{H^r}\simeq M$. By definition of $b$ we have that
\begin{equation}
\| b_3(t,\cdot) - \langle \tilde{b}_3 \rangle  \|_{L^\infty}\leq \tilde C(\|m(t,\cdot)\|_{H^r}+\|\rho(t,\cdot)\|_{H^r}),
\end{equation}
with $\tilde C$ being the constant in the Sobolev embedding. We now show the following estimate
\begin{equation}\label{stima capocchia}
\|\rho(t,\cdot)\|_{H^s}^2+\eta\int_0^t\|\rho(\tau,\cdot)\|_{H^{s+1}}^2\de\tau\leq C \frac{M^2}{\eta^s}R^s.
\end{equation}
We proceed by induction on integer values of $s$: for $s=0$ the above inequality is the standard energy estimate. We assume that \eqref{stima capocchia} holds for $s-1$. Then, we apply the operator $\langle D\rangle^s$ to the equation and we use that $U$ is divergence-free to deduce
\begin{align*}
\frac12\frac{\de}{\de t}\|\rho(t,\cdot)\|_{H^s}^2+\eta\|\rho(t,\cdot)\|_{H^{s+1}}^2&\leq \int_{\T^2}|[U(t,\cdot), \langle D\rangle^s]\rho(t,x)||\nabla \langle D\rangle^s\rho(t,x)|\de x\\
&\leq \frac{\eta}{2}\|\rho(t,\cdot)\|_{H^{s+1}}^2\\
&+\frac{2}{\eta}\left(\|\nabla U(t,\cdot)\|_\infty\|\rho(t,\cdot)\|_{H^{s-1}}+\|U(t,\cdot)\|_{H^s}\|\rho(t,\cdot)\|_\infty\right)^2
\end{align*}
where in the last inequality we used Young's inequality and Lemma \ref{lem:kato ponce}. Then, we integrate in time and we use the induction hypothesis and (for any $s\leq r+1$) \eqref{restrizione 1} to get
\begin{equation}
\|\rho(t,\cdot)\|_{H^s}^2+\eta\int_0^t\|\rho(t,\cdot)\|_{H^{s+1}}^2\leq \frac{2M^2}{\eta}R\left(1+\frac{R^{s-1}}{\eta^{s-1}}\right),
\end{equation}
which implies the result. Notice that, if we do not use the viscosity, the estimate is exponential in $R$. From Theorem B' we also have that
\begin{equation}\label{confronto stoc}
\|\rho(t,\cdot)\|_{H^s}^2\leq CM^2 \frac{R^s}{\eta^s}\frac{e^{-2\hat \gamma t}}{\eta^2},
\end{equation}
and then, we observe that the estimate in \eqref{stima capocchia} is better than the one in \eqref{confronto stoc} if $t\lesssim |\log \eta|$.\\

Now we want to choose $\e$ in order to apply, in a convenient way, Theorem \ref{thm:stab} together with Remark \ref{rem:stabilità viscoso}.
We set $R=T^{2(r+2)}$ and we consider $T\sim \frac{1}{\eta}$, so that by taking
\begin{equation}\label{epsilon stocastico}
\e\leq \frac{CM}{2\tilde C}\eta^{\frac{r+1}{2}}e^{-\frac{CM}{\eta^{\frac32(r+1)}}},
\end{equation}
we have that the solution is defined up to $T\sim\frac{1}{\eta}$ and
\begin{equation}\label{stima m stocastica 0}
\|m(t,\cdot)\|_{H^r}\leq \frac{M}{2\tilde C}.
\end{equation}
Lastly, since $\rho$ satisfies \eqref{eq:stima l2 bedrossian}, again by Theorem B' with $r>5/2$ and $\Lambda(\eta)=\hat{\gamma}$ to obtain that
\begin{equation} \label{stima rho stocastica}
\|\rho(t,\cdot)\|_{H^r}\leq \frac{C R^{\frac{r}{2}}}{\eta^\frac{r+3}{2}}e^{-\hat{\gamma} t}\|\rho^\mathrm{in}\|_{H^r}.
\end{equation}
Then, by defining $t^*$ as
\begin{equation}
t^{*} = \frac{1}{\hat\gamma}\ln\left(\frac{2 CR^{\frac{r}{2}}\tilde C }{\eta^{\frac{r+3}{2}}}\right),
\end{equation}
we have that for any $t\geq t^*$ it holds
\begin{equation}
\tilde C \| \rho(t,\cdot) \|_{H^r} \leq   \frac{M}{2}.
\end{equation}
In conclusion, considering our choice of $R$ and $T$, we have obtained that for $1/\lambda_\mathrm{rec}(\eta)\lesssim t<T$, with
$$
\lambda_\mathrm{rec}(\eta):=\frac{C(r,\hat\gamma)}{|\ln\eta|},
$$
the magnetic field does not have any zero. This concludes the proof.
\end{proof}

\subsection*{Acknowledgements}
G.C. is supported by INdAM-GNAMPA and by the projects PRIN2020 ``Nonlinear evolution PDEs, fluid dynamics and transport equations: theoretical foundations and applications” and PRIN2022 ``Classical equations of compressible fluid mechanics: existence and properties of non-classical solutions''. 
R.L. is supported by the Basque Government through the program BERC 2022-2025 (BCAM), by the project PID2021-123034NB-I00 funded by MCIN/ AEI /10.13039/501100011033, by the Severo Ochoa accreditation CEX2021-001142-S (BCAM) and by the Ramon y Cajal fellowship RYC2021-031981-I.\\
The authors would like to thank Emanuele Tassi for his insightful comments and for providing relevant literature on magnetic reconnection, as well as Michele Coti Zelati, Michele Dolce, Lucas Kaufmann and Lorenzo Pescatore for their valuable discussions on the topics addressed in this work. 

\section*{Data Availability Statement}
This manuscript has no associated data.

\end{document}